\documentclass[review,hidelinks,onefignum,onetabnum]{siamart220329}

\usepackage{appendix}
\usepackage{lipsum}
\usepackage{amsfonts}
\usepackage{graphicx}
\usepackage{epstopdf}
\usepackage{algorithmic}
\ifpdf
  \DeclareGraphicsExtensions{.eps,.pdf,.png,.jpg}
\else
  \DeclareGraphicsExtensions{.eps}
\fi

\newcommand{\tr}{\mathrm{tr}}        %trace
\newcommand{\diam}{\mathrm{diam}}       %set diameter
\newcommand{\dd}{:} % matrix dot product
\newcommand{\vnu}{\bm{\nu}}

\newcommand{\vS}{\mathbf{S}}

\newcommand{\vV}{\mathbf{V}}

\newcommand{\vH}{\mathbf{H}}

\newcommand{\Om}{\Omega}

\newcommand{\Gm}{\Gamma}
\newcommand{\iO}{\int_{\Omega}}

\newcommand{\iG}{\int_{\Gamma}}

\newcommand{\M}{\mathbb{M}}   % another space
\newcommand{\R}{\mathbb{R}}   % real numbers
\newcommand{\bigW}{\mathcal{W}} % the big W space
\newcommand{\Vs}{\vV}
\newcommand{\Hs}{\vH}
\newcommand{\HGs}{\vH_{\Gamma}}
\newcommand{\HGsh}{\vH_{\Gamma,h}}

\newcommand{\dt}{\delta t}

\newcommand{\inner}[3]{\left( #1, #2 \right)_{#3}}
\newcommand{\Tk}{\mathcal{T}}
\newcommand{\Fk}{\mathcal{F}}

\newcommand{\Pk}{\mathcal{P}}
\newcommand{\interp}{\mathcal{I}_{h}} % interpolation

\newcommand{\symmten}{\vS} % symmetric tensors
\newcommand{\symmtraceless}{\vS_{0}} % symmetric and traceless
\newcommand{\tracelessop}[1]{\widehat{#1}} % symmetric and traceless
\newcommand{\tA}{A}

\newcommand{\tQ}{Q}
\newcommand{\tP}{P}
\newcommand{\tH}{H}
\newcommand{\tR}{R}

\newcommand{\tX}{X}
\newcommand{\tY}{Y}

\newcommand{\tXi}{\Xi}
\newcommand{\tzero}{0}
\newcommand{\basis}{E}
\newcommand{\ELdG}{\mathcal{E}}
\newcommand{\ELdGfunc}{W}

\newcommand{\LdGsurf}{f_{\Gm}}
\newcommand{\anchcoef}{\eta_{\Gm}}
\newcommand{\LdGrhs}{\chi}
\newcommand{\Li}{\ell}

\newcommand{\bulkfunc}{\psi}

\newcommand{\bulkimp}{\psi_{c}}
\newcommand{\bulkexp}{\psi_{e}}
\newcommand{\bulketa}{\eta_{\mathrm{dw}}}
\newcommand{\bulkK}{a_0}
\newcommand{\bulkA}{a_2}
\newcommand{\bulkB}{a_3}
\newcommand{\bulkC}{a_4}
\newcommand{\bulkstab}{d_2}
\newcommand{\bulkbnd}{\tilde{d}_2}
\newcommand{\domcyl}{\mathcal{C}} % domain cylinder
\newcommand{\bdycyl}{\mathcal{G}} % bdy cylinder
\newcommand{\tfinal}{t_{\mathrm{f}}}

\newcommand{\domtarget}{Z_{\domcyl}} % target domain function
\newcommand{\bdytarget}{Z_{\bdycyl}} % target bdy function
\newcommand{\fintarget}{Z_{\tfinal}} % target final time function

\newcommand{\domcon}{U_{\Om}} % domain control
\newcommand{\domconi}[1]{U_{\Om,#1}} % domain control
\newcommand{\domconh}{U_{\Om,h}} % domain control (discrete)
\newcommand{\domcoef}{\eta_{\Om}} % coefficent for domain control
\newcommand{\bdycon}{U_{\Gm}} % boundary control
\newcommand{\bdyconi}[1]{U_{\Gm,#1}} % boundary control
\newcommand{\bdyconh}{U_{\Gm,h}} % boundary control
\newcommand{\bardomcon}{\bar{U}_{\Om}} % optimal boundary control
\newcommand{\barbdycon}{\bar{U}_{\Gm}} % optimal boundary control

\newcommand{\dombnd}{{u}_{\Om}}

\newcommand{\conadmis}{\mathcal{U}_{\rm{ad}}}
\newcommand{\conreg}{\mathcal{U}_{\rm{reg}}}
\newcommand{\tilconadmis}{\widetilde{\mathcal{U}}_{\rm{ad}}}

\newcommand{\costobj}{J}
\newcommand{\reducedobj}{\mathcal{J}}

\usepackage{lineno,hyperref}
\modulolinenumbers[5]
\usepackage{amssymb}
\usepackage{amsmath}
\usepackage{bm}
\usepackage{color}      % colored text
\usepackage{enumerate}
\usepackage{stackengine}
\usepackage{subcaption}
\newsiamremark{remark}{Remark}
\newsiamremark{hypothesis}{Hypothesis}
\crefname{hypothesis}{Hypothesis}{Hypotheses}
\newsiamthm{claim}{Claim}
\newsiamthm{assumption}{Assumption}

\headers{Optimal Control of the Landau-de Gennes Model}{T. M. Surowiec and S. W. Walker}

\title{Optimal Control of the Landau-de Gennes Model \\ of Nematic Liquid Crystals\thanks{Submitted to the editors DATE.
}}
\author{Thomas M. Surowiec\thanks{
Simula Research Laboratory, 
Department of Numerical Analysis and Scientific Computing,
Kristian Augusts gate 23,
0164 Oslo, Norway
  (\email{thomasms@simula.no}).}
\and Shawn W. Walker\thanks{
Department of Mathematics, 
Center for Computation and Technology (CCT),
Louisiana State University,
Baton Rouge, LA 70803,
Tel.: +1-225-578-1603,
Fax: +1-225-578-4276,
S. W. Walker acknowledges financial support by the NSF DMS-2111474.
  (\email{walker@math.lsu.edu}).}
}

\usepackage{amsopn}
\ifpdf
\hypersetup{
  pdftitle={Optimal Control of Landau-de Gennes},
  pdfauthor={T. M. Surowiec and S. W. Walker}
}
\fi

\begin{document}
\nolinenumbers
\maketitle

% REQUIRED
\begin{abstract}
We present an analysis and numerical study of an optimal control problem for the Landau-de Gennes (LdG) model of nematic liquid crystals (LCs), which is a crucial component in modern technology. They exhibit long range orientational order in their nematic phase, which is represented by a tensor-valued (spatial) order parameter $Q = Q(x)$.  Equilibrium LC states correspond to $Q$ functions that (locally) minimize an LdG energy functional.  Thus, we consider an $L^2$-gradient flow of the LdG energy that allows for finding local minimizers and leads to a semi-linear parabolic PDE, for which we develop an optimal control framework.  We then derive several a priori estimates for the forward problem, including continuity in space-time, that allow us to prove existence of optimal boundary and external ``force'' controls and to derive optimality conditions through the use of an adjoint equation. Next, we present a simple finite element scheme for the LdG model and a straightforward optimization algorithm.  We illustrate optimization of LC states through numerical experiments in two and three dimensions that seek to place LC defects (where $Q(x) = 0$) in desired locations, which is desirable in applications.
\end{abstract}

% REQUIRED

% REQUIRED

%%%%%%%%%%%%%%%%%%%%%%%%%%%%%%%%%%%%%%%%%%%%%
\section{Introduction}\label{sec:intro}
%%%%%%%%%%%%%%%%%%%%%%%%%%%%%%%%%%%%%%%%%%%%%
Liquid crystals (LCs) are a critical material for emerging technologies \cite{deGennes_book1995,Lagerwall_CAP2012}.  Their response to optical \cite{Blinov_book1983,Goodby_inbook2012,Sun_SMS2014,Hoogboom_RSA2007,Dasgupta_MRE2015}, electric/magnetic \cite{Brochard_JPhysC1975,Buka_book2012,Shah_Small2012}, and mechanical actuation \cite{Zhu_PRE2011,LC_Elastomers_book2012,Biggins_JMPS2012,Resetic_NC2016} has already yielded various devices, e.g. electronic shutters \cite{Heo_AA2015}, novel types of lasers \cite{Humar_OE2010,Coles_NP2010}, dynamic shape control of elastic bodies \cite{Camacho-Lopez_NM2004,Ware_Science2015}, and others \cite{Musevic2011,Lopez-Leon_CPS2011,Copar_PNAS2015,Whitmer_PRL2013,Wang_NL2014}.  In fact, \cite{dePablo_SA2022} demonstrate that LCs can enable logic operations within soft matter, which can lead to creating autonomous active materials with the capability to make decisions.  Thus, optimization of LC devices in these applications is of obvious interest.

LCs are considered a \emph{meso-phase} of matter in which its ordered macroscopic state is \emph{between} a spatially disordered liquid, and a fully crystalline solid \cite{Virga_book1994}.  In their nematic phase, in which long ranged orientational order exists, the Landau-de Gennes (LdG) theory introduces a \emph{tensor-valued} function $\tQ$ to describe local order in the LC material.  In particular, the eigenframe of $\tQ$ yields information about the statistics of the distribution of LC molecule orientations; {see \cite[Sec. 1.3]{Virga_book1994} for an excellent derivation}.  The energy functional for $\tQ$, which is minimized at equilibrium, involves both a bulk potential, of ``double-well'' type, and an elastic contribution involving the derivatives of $\tQ$.  Often, an $L^2$-type gradient flow is used to compute (local) minimizers of the LdG energy functional.

The goals of this paper are to formulate an optimal control problem for the $L^2$-gradient flow of the LdG energy, derive several analytical results, and demonstrate the ability to optimize LC behavior with numerical simulations. To the best of our knowledge, a fully fledged, PDE-based, optimal control formulation of the LdG model of LCs has not been done before.  Utilizing both boundary controls and external ``force'' controls, we prove existence of optimal controls for the LdG model.  In addition, we show several numerical experiments, of tracking control type, that seek to place LC \emph{defects} in desired locations.  {Defects correspond to sudden spatial changes in $\tQ$ and are discussed more thoroughly in \cref{sec:results}; also see \cite{LinLiu_JPDE2001} for an introduction to defects in mathematical models of liquid crystals.}  Our method should be useful for optimizing LC devices in a variety of applications.

The paper is organized as follows.  \Cref{sec:lct} explains the LdG model and the associated optimal control problem, as well as discuss related work on the Allen-Cahn equation.  The well-posedness of the parabolic PDE coming from the $L^2$-gradient flow of the LdG energy is established in \cref{sec:well_psd} along with several analytical results.  Existence of optimal controls is shown in \cref{sec:exist_ctrl} and first order optimality conditions are established in \cref{sec:first_order}.  \Cref{sec:FEM_approx} describes our finite element method for approximating the forward and adjoint problems; see \cite{Bajc_JCP2016,Davis_SJNA1998,Ravnik_LC2009,Bartels_bookch2014,BorthNochettoWalker_NM2020,Lee_APL2002,Zhao_JSC2016} for other numerical methods for models related to LdG.  We illustrate our method with numerical experiments in \cref{sec:results} and close with some remarks in \cref{sec:conclusion}. 

%%%%%%%%%%%%%%%%%%%%%%%%%%%%%%%%%%%%%%%%%%%%%
\section{Liquid Crystal Theory}\label{sec:lct}
%%%%%%%%%%%%%%%%%%%%%%%%%%%%%%%%%%%%%%%%%%%%%

This section reviews the Landau-de Gennes (LdG) theory for a nematic phase.  We start with the following clarifications.

{First, we note that the minimization of the standard free energy of the LdG model gives rise to a semilinear elliptic partial differential equation (PDE), which admits multiple solutions.  Theoretically, this semilinear equation could be included as a constraint in an optimization problem. However, there is no guarantee that the second derivative of the LdG free energy functional is surjective. This would significantly complicate the derivation of optimality conditions and severely limit the convergence theory of numerical optimization algorithms.}  
%Afterwards, we define an optimal control problem, for which the LdG model, a semilinear elliptic partial differential equation (PDE), arises as a constraint.  
%However, the stationary LdG model does not admit unique solutions. Although full space approaches for PDE-constrained optimization exist in which the LdG model would be left as an equality constraint,
%would not be implicitly treated with a well-defined control-to-state mapping, exist, 
%there is no guarantee that the Hessian of the LdG free energy functional is surjective. 
To remedy these issues, we will consider an evolution equation, which amounts to an $L^2$-gradient flow of the LdG free energy.  This time-dependent control strategy is analyzed in subsequent sections.

{The second point of clarification involves the bulk potential used to model the nematic-to-isotropic phase transition $\widetilde{\bulkfunc}$. In the discussion below, we will first introduce a traditional double well function and derive an associated evolution equation. For mathematical reasons, we then modify this term \textit{beyond physically meaningful values} of $|Q|$.}

\paragraph{Notation} We typically denote scalars and vectors with lowercase letters, while tensors are denoted with uppercase letters.  Boldface capital letters typically denote vector spaces or function spaces.  Standard notation is used for Sobolev spaces and inner products.

%%%%%%%%%%%%%%%%%%%%%%%%%%%%%%%%%%%%%%%%%%%%%
\subsection{Landau-de Gennes Model}\label{ssec:ldg}
%%%%%%%%%%%%%%%%%%%%%%%%%%%%%%%%%%%%%%%%%%%%%
Let $\symmten$ be the space of symmetric, $d \times d$ tensors, and $\symmtraceless$ the set of symmetric, traceless $d \times d$ tensors, where $d = 2$ or $3$.  The order parameter of the LdG theory is given by $\tQ \in \symmtraceless$, which represents the statistical distribution (i.e.\ a covariance matrix) of LC molecules at a given point in space \cite{Virga_book1994}.  This means that the eigenvalues $\lambda_{i} (\tQ)$ should satisfy the following bound: $-1/d \leq \lambda_{i} (\tQ) \leq (d-1)/d$ for $i=1,...,d$.  In the standard model discussed below, the eigenvalue bounds are not explicitly enforced, though they are usually satisfied through the effect of the double-well (c.f. \cite{Majumdar_EJAM2010}).

We mainly focus on the $d=3$ case, and represent the state of the LC material by a tensor-valued function $\tQ : \Om \to \symmtraceless$, where $\Om \subset \R^3$ is the physical domain of interest.  Moreover, we take $\Om$ to be an open, bounded, Lipschitz domain with boundary $\Gm$, and outward unit normal $\vnu$; the normal derivative is denoted $\partial_{\vnu}$.  The standard free energy of the LdG model is defined as \cite{Mori_JJAP1999,Mottram_arXiv2014}:
\vspace{-0.4cm}
\begin{equation}\label{eqn:Landau-deGennes_general}
\begin{split}
  \ELdG [\tQ] &:= \iO \ELdGfunc (\tQ,\nabla \tQ) \, dx +
  \frac{1}{\bulketa^2}
  \iO \widetilde{\bulkfunc} (\tQ) \, dx + \anchcoef \iG \LdGsurf(\tQ) \, dS(x) - \iO \LdGrhs(\tQ) \, dx,
\end{split}
\end{equation}
where $\ELdGfunc (\tQ,\nabla \tQ)$ is the elastic energy density \cite{Mori_JJAP1999,Mottram_arXiv2014}.  Since optimal control of the Landau-de Gennes model has yet to be developed, we simplify $\ELdGfunc (\tQ,\nabla \tQ)$ to the one-constant model, i.e. $\ELdGfunc (\tQ,\nabla \tQ) = \frac{1}{2} |\nabla \tQ|^2$.  Future extensions of this work will consider more general elastic energies.

%\begin{equation}\label{eqn:Landau-deGennes_energy_density}
%\begin{split}
%\ELdGfunc(\tQ,\nabla \tQ) &:= \frac{1}{2} \Big{(} \Li_{1} |\nabla \tQ|^2 + \Li_{2} |\nabla \cdot \tQ|^2 + \Li_{3} (\nabla \tQ)\tp \trid \nabla \tQ , \\
%&\qquad\qquad + \Li_{4} \nabla \tQ \trid (\levi \cdot \tQ) + \Li_{*} \nabla \tQ \trid [(\tQ \cdot \nabla) \tQ] \Big{)},
%\end{split}
%\end{equation}
%where $\{ \Li_{i} \}_{i=1}^{4}$, $\Li_*$, are material dependent elastic constants, and 
%\begin{equation}\label{eqn:Landau-deGennes_invariants}
%\begin{split}
%|\nabla \tQ|^2 := (\partial_{k} Q_{ij})^2, \quad 
%|\nabla \cdot \tQ|^2 := (\partial_{j} Q_{ij})^2, \quad (\nabla \tQ)\tp \trid \nabla \tQ &:= (\partial_{j} Q_{ik}) (\partial_{k} Q_{ij}), \\
%\nabla \tQ \trid (\levi \cdot \tQ) := \levi_{j k l} Q_{ji} \partial_{l} Q_{ki}, \quad \nabla \tQ \trid [(\tQ \cdot \nabla) \tQ] := Q_{lk} & (\partial_{l} Q_{ij}) (\partial_{k} Q_{ij}),
%\end{split}
%\end{equation}
%where we use the convention of summation over repeated indices.  The five elastic constants are related to the five independent constants of the Oseen-Frank model \cite{Mori_JJAP1999,Mottram_arXiv2014}.  Indeed, $\Li_{4}$ accounts for cholesteric ``twist'' and $\Li_{*}$ is needed to have five independent constants.  Since optimal control of the Landau-de Gennes model has yet to be developed, we simplify \cref{eqn:Landau-deGennes_energy_density} to the one-constant LdG model, i.e. $\Li_{1} = 1$ and $\Li_{i} = 0$, for $i=2,3,4$, and $\Li_{*} = 0$.  Future extensions of this work will consider the more general case.

The bulk potential $\widetilde{\bulkfunc}$ models the nematic-to-isotropic phase transition.  It is a (non-symmetric) double-well type of function that is typically given by
\vspace{-0.1cm}
\begin{equation}\label{eqn:Landau-deGennes_bulk_potential}
\begin{split}
\widetilde{\bulkfunc} (\tQ) = \bulkK - \frac{\bulkA}{2} \tr (\tQ^2) - \frac{\bulkB}{3} \tr (\tQ^3) + \frac{\bulkC}{4} \left( \tr (\tQ^2) \right)^2, \quad \widetilde{\bulkfunc} \geq 0,
\end{split}
\end{equation}
where $\bulkA$, $\bulkB$, $\bulkC$ are material parameters.  The choice of constants affects the stability of the nematic phase.  Since we are only interested in the nematic phase, $\bulkA$, $\bulkB$, $\bulkC$ are positive, and $\bulkK$ is a convenient positive constant to ensure $\widetilde{\bulkfunc} \geq 0$.  
Stationary points of $\widetilde{\bulkfunc}$ are either \emph{uniaxial} or \emph{isotropic} \cite{Majumdar_EJAM2010}, i.e. if $\tQ$ is uniaxial, then it corresponds to
\vspace{-0.2cm}
\begin{equation}\label{eqn:uniaxial}
\begin{split}
    \tQ_{ij} = s_0 \left( n_i n_j - \delta_{ij}/3 \right), \text{ for } 1 \leq i,j \leq 3,
\end{split}
\end{equation}
where $s_0 > 0$ depends on $\widetilde{\bulkfunc}$, $[n_{i}]_{i=1}^3 \equiv n \in \R^3$ is a unit vector, and $\delta_{ij}$ is the Kronecker delta; $\tQ = \tzero$ is the isotropic state.  The parameter $\bulketa > 0$, appearing in \cref{eqn:Landau-deGennes_general}, is known as the \emph{nematic correlation length}, and usually satisfies $\bulketa \ll 1$.

The surface energy $\LdGsurf(\tQ)$, with parameter $\anchcoef > 0$, accounts for \emph{weak anchoring} of the LC material at the boundary, i.e. it imposes an energetic penalty on the boundary conditions for $\tQ$.  In this paper, we use a Rapini-Papoular type anchoring energy \cite{Barbero_JPF1986}:
\vspace{-0.2cm}
\begin{equation}\label{eqn:Landau-deGennes_surf_energy}
\begin{split}
\LdGsurf (\tQ) = \frac{1}{2} \tr \left( \tQ - \bdycon \right)^2 \equiv \frac{1}{2} |\tQ - \bdycon|^2,
\end{split}
\end{equation}
where $\bdycon : \Gm \to \symmtraceless$ and we take $\bdycon$ to be one of the control variables in the optimal control problem stated in \cref{sec:optim_ctrl_prob}.  
The function $\LdGrhs(\cdot)$ is used to (approximately) model interactions of the LC material with external fields, e.g. an electric field.  In this paper, we take $\LdGrhs(\tQ) = \tQ \dd \domcon$, where $\domcon : \Om \to \symmtraceless$ is also a control variable.  

Local minimizers of the energy $\ELdG [\tQ]$ can be found through an $L^2$ gradient flow, which can be thought of as a simple damped, evolutionary LdG model.  This leads to the following parabolic equation for $\tQ$ in strong form:
\vspace{-0.2cm}
\begin{subequations}\label{eq:forward_problem}
\begin{align}
\tQ_t - \Delta \tQ + \frac{1}{\bulketa^2} \tracelessop{\widetilde{\bulkfunc}'(\tQ)}
&= \domcoef \domcon, \text{ in } \Om \times (0,\tfinal), \\
\partial_{\vnu} \tQ + \anchcoef \tQ
&= 
\anchcoef \bdycon,  \text{ on } \Gm \times (0,\tfinal), \\
\tQ(\cdot,0) &= \tQ_0, \quad ~ \text{ in } \Om,
\end{align}
\end{subequations}
where $\tracelessop{M}$ denotes the traceless part of a symmetric tensor $M$, $\tQ_0 : \Om \to \symmtraceless$ is the initial condition, and $\tfinal > 0$ is a given final time.  The system \cref{eq:forward_problem} can be viewed as a tensor-valued analog of the Allen-Cahn equation with Robin boundary conditions.  Taking $\anchcoef \to \infty$ recovers \emph{strong anchoring}, i.e. the Dirichlet condition $\tQ = \bdycon$ on $\Gm$.  In this paper, $\anchcoef > 0$ is fixed and finite.

The first and second derivatives of $\widetilde{\bulkfunc}$ are a 2-tensor and 4-tensor, respectively:
\vspace{-0.2cm}
\begin{equation}\label{eqn:1st_2nd_deriv_bulkfunc}
\begin{split}
    \widetilde{\bulkfunc}' (\tQ) &= -\bulkA \tQ - \bulkB \tQ^2 + \bulkC \, \tr (\tQ^2) \tQ, \\
    [\widetilde{\bulkfunc}'' (\tQ)]_{ijkl} &= -\bulkA \delta_{ik} \delta_{jl} - 2 \bulkB \tQ_{jk} \delta_{li} + \bulkC \left( 2 \tQ_{ij} \tQ_{kl} + \tr (\tQ^2) \delta_{ik} \delta_{jl} \right),
\end{split}
\end{equation}
where the second derivative is written with indices for clarity.  In addition, for analytical purposes, we modify the bulk potential to have quadratic growth as $|\tQ| \to \infty$.  For instance, let $\rho : [0, \infty) \to \R_{+}$ be a $C^{\infty}$ cut-off function such that
\vspace{-0.2cm}
\begin{equation*}%\label{eqn:smooth_cut-off}
\begin{split}
	\rho(r) = 1, \text{ if } r < b_1, ~~ \rho(r) = \text{monotone}, \text{ if } b_1 \leq r \leq b_2, ~~ \rho(r) = 0, \text{ if } r > b_2,
\end{split}
\end{equation*}
where $1 \leq b_1 < b_2$ are two fixed constants.  Then, the modified potential is given by
\vspace{-0.2cm}
\begin{equation}\label{eqn:mod_LdG_bulk_potential}
\begin{split}
\bulkfunc (\tQ) = \widetilde{\bulkfunc} (\tQ) \rho(\tr (\tQ^2)) + \bulkC^2 \tr (\tQ^2) \left[ 1 - \rho(\tr (\tQ^2)) \right],
\end{split}
\end{equation}
which will be used throughout the remainder of the paper.  
Clearly, there exist uniform constants $c_0$, $c_1$, $c_2$, depending on $\bulkK$, $\bulkA$, $\bulkB$, $\bulkC$, $b_1$, $b_2$, such that
\begin{equation}\label{eqn:mod_LdG_bulk_pot_bnds}
\begin{split}
	|\bulkfunc (\tQ)| \leq \bulkK + c_0 |\tQ|^2, \quad |\bulkfunc' (\tQ)| \leq c_1 |\tQ|, \quad |\bulkfunc'' (\tQ)| \leq c_2, ~ \text{ for all } \tQ \in \symmtraceless,
\end{split}
\end{equation}
noting that $|\tQ|^2 \equiv \tr (\tQ^2)$.  For typical choices of the physical constants in \cref{eqn:Landau-deGennes_bulk_potential}, choosing $b_1 = 1$ and $b_2 = 2$ is effective since this modification does not change the location of the local minimizers. In addition to \cref{eqn:mod_LdG_bulk_pot_bnds} we observe that $\psi''$ is globally Lipschitz with uniformly bounded derivative $\psi'''(\tQ)$ for all $\tQ \in \symmtraceless$.

In addition, we will often make use of the convex splitting
\vspace{-0.2cm}
\begin{equation}\label{eqn:bulkfunc_convex_split}
\begin{split}
	\bulkfunc(\tQ) &\equiv \left[ (\bulkstab/2) \tr (\tQ^2) + \bulkfunc (\tQ) \right] - (\bulkstab/2) \tr (\tQ^2) =: \bulkimp(\tQ) - \bulkexp(\tQ),
\end{split}
\end{equation}
where $\bulkimp, \bulkexp$ are non-negative convex functions with $\bulkstab > 0$ (i.e. a ``stabilization'' constant) chosen sufficiently large to ensure a convex split.  In particular, we note that $\bulkimp'(\tQ)$ and $\bulkexp'(\tQ)$ are monotone functions and there is a constant $0 < \bulkbnd < \bulkstab$ such that, for $\bulkstab$ sufficiently large, $\bulkimp(\tQ)$ satisfies the lower bounds:
\vspace{-0.2cm}
\begin{equation}\label{eqn:bulkimp_lower_bnd}
\begin{split}
	\bulkimp(\tQ) \geq \bulkK + \frac{\bulkbnd}{2} |\tQ|^2, \quad \bulkimp'(\tQ) \dd \tQ \geq \bulkbnd |\tQ|^2, \quad \tP \dd \bulkimp''(\tQ) \dd \tP \geq 3 \bulkbnd |\tP|^2.
\end{split}
\end{equation}

Henceforth, we take all constants to be non-dimensional; see \cite{Gartland_MMA2018} for a detailed treatment of how the LdG model is non-dimensionalized.

%%%%%%%%%%%%%%%%%%%%%%%%%%%%%%%%%%%%%%%%%%%%%
\subsection{Optimal control problem}\label{sec:optim_ctrl_prob}
%%%%%%%%%%%%%%%%%%%%%%%%%%%%%%%%%%%%%%%%%%%%%
We formulate an optimal tracking control problem for the LdG model.  The following Sobolev spaces are used throughout:
\vspace{-0.2cm}
\begin{equation*}
\begin{split}
    \Vs := H^1(\Om;\symmtraceless),\quad \Hs := L^2(\Om;\symmtraceless), \quad \HGs := L^2(\Gm;\symmtraceless),
\end{split}
\end{equation*}
where each space is endowed with its respective natural norm.  For every $t \in (0, \tfinal]$, the space-time cylinder and boundary are defined as
\vspace{-0.2cm}
\begin{equation}\label{eqn:space_time_cylinder}
\begin{split}
    \domcyl_{t} &:= \Om \times (0,t), ~~ \domcyl \equiv \domcyl_{\tfinal}, ~~ \bdycyl_{t} := \Gm \times (0,t), \text{ and } \bdycyl \equiv \bdycyl_{\tfinal}.
\end{split}
\end{equation}
Next, we introduce target functions:
\vspace{-0.2cm}
\begin{equation}\label{eq:targets_bounds}
\begin{split}
\domtarget &\in L^{2}(\domcyl), \quad
\bdytarget \in L^{2}(\bdycyl), \quad
\fintarget \in H^1(\Omega).
\end{split}
\end{equation}
In contrast to optimal control problems with scalar or vector-valued controls, the bound constraints used to defined the set of admissible controls $\conadmis$ are slightly more complicated. These are discussed below.

We now define the optimal control problem: minimize the functional
\vspace{-0.2cm}
\begin{equation}\label{eq:optconprob}
\begin{split}
\costobj(\tQ,\domcon,\bdycon) &:=
\frac{\beta_{\domcyl}}{2} \| \tQ - \domtarget \|^2_{L^2(\domcyl)} + 
\frac{\beta_{\bdycyl}}{2} \| \tQ - \domtarget \|^2_{L^2(\bdycyl)} \\
+ \frac{\beta_{\tfinal}}{2} & \| \tQ(\cdot,\tfinal) - \fintarget \|^2_{\Hs} + \frac{\alpha_{\domcyl}}{2} \| \domcon \|^2_{L^2(\domcyl)} + \frac{\alpha_{\bdycyl}}{2}  \| \bdycon \|^2_{H^1(0,\tfinal;\HGs)},
\end{split}
\end{equation}
over a set of admissible controls $(\domcon,\bdycon) \in \conadmis \subset L^2(\domcyl) \times L^2(\bdycyl)$, subject to $\tQ$ satisfying the PDE constraint \cref{eq:forward_problem}.  The coefficients satisfy $\beta_{\domcyl}, \beta_{\bdycyl}, \beta_{\tfinal}, \alpha_{\domcyl}, \alpha_{\bdycyl} \ge 0$, where at least one of $\beta_{\domcyl}, \beta_{\bdycyl}, \beta_{\tfinal}$ is nonzero, and $\alpha_{\domcyl}, \alpha_{\bdycyl} > 0$. Tracking objectives as in \cref{eq:optconprob} are ubiquitous in optimal control. For our application, the final two summands are quadratic cost functionals that force a certain regularity. The first two terms represent the desire to track a transition of targeted textures  $\domtarget$, whereas the third summand $\fintarget$ is a desired stationary texture. Nematic textures correspond to director orientations associated with $\tQ$ at each point $x$ in the domain. Therefore, a desired texture $\fintarget$ could be one in which all directors are oriented in the same direction, e.g., parallel to a surface, or in which the directors follow a particular pattern. 

In practice, the available control mechanisms may be technically limited, e.g., finite dimensional, stationary or only on the boundary. We include distributed controls in the bulk and boundary here for more generality. Restrictions to the cases just mentioned would not change the core of the analysis.  The control set $\conadmis$ is always taken to be nonempty, closed, and convex. An example of such a set is
\vspace{-0.3cm}
\begin{equation}\label{eq:con_constr}
\begin{split}
\conadmis := &\left\{ \tP \in L^2(\domcyl) \mid 
|\tP| \le \dombnd,
\text{ a.e. in } \domcyl \right\} \times \left\{ \tP \in H^1(0,\tfinal;\HGs) \mid 
|\tP| \le 1, \text{ a.e. in } \bdycyl \right\}.
\end{split}
\end{equation}
Here, $\dombnd$ may be an arbitrary, essentially bounded scalar-valued function on $\domcyl$. The constant bound $1$ for the boundary controls is in fact dictated by the application (recall the eigenvalue bounds discussed in \cref{ssec:ldg}).  Note that if $\bdycon$ is constant in time, then the $H^1(0,\tfinal;\HGs)$ space is replaced by $\HGs$, and the $\alpha_{\bdycyl}$ term in \cref{eq:optconprob} becomes an $\HGs$ norm.  In most applications, boundary controls are \emph{constant} in time.  Allowing the boundary control to vary in time is mainly for the sake of generality.

Due to the similarities of \cref{eq:forward_problem} with the Allen-Cahn equation, there are a wide array of relevant contributions in the literature, where optimal control of Allen-Cahn and related equations, e.g., Cahn-Hilliard, have been studied. We highlight here several early studies \cite{Hoffman_NFAO1992,Heinkenschloss_OMS1997,Heinkenschloss_CC1999}, which focused on the optimal control of Cahn-Hilliard (phase field problems of Caginalp-type); more recent work \cite{FarshbafShaker_NFAO2012,FarshbafShaker_AMO2015}, in which the author studied the optimal control of scalar- and vector-valued Allen-Cahn equations with a nonsmooth bulk energy term (obstacle potential), and \cite{Colli_SJCO2015}.  In some sense, \cite{Colli_SJCO2015} is the most relevant.  However, there are several major differences.  Our boundary condition has no diffusive term, because it is not clear how that would manifest in an LC system, which thus affects our solution's regularity.  Moreover, we are dealing with a parabolic \emph{system} with \emph{tensor-valued} solutions and controls; the PDE in \cite{Colli_SJCO2015} is scalar-valued.  This affects several arguments needed to derive first-order optimality conditions and greatly increases the difficulty for numerical methods. 

%---------------------------------------------------------------
\section{Well-posedness of the forward problem}\label{sec:well_psd}
%---------------------------------------------------------------

To prove well-posedness of \cref{eq:forward_problem}, we start with the usual arguments (cf. \cite{Evans:book}).  The minimal regularity of the data is given by
\vspace{-0.2cm}
\begin{equation}\label{eq:minimal_input_reg}
\domcon \in L^2(0,\tfinal;\Hs),\quad
\bdycon \in H^1(0,\tfinal;\HGs)\quad
\tQ_0 \in \Vs,
\end{equation}
and the space of weak solutions we consider is
\vspace{-0.1cm}
\begin{equation}\label{eq:space_of_weak_soln}
	\bigW := L^{\infty}(0,\tfinal;\Vs) \cap L^{\infty}(0,\tfinal;\HGs) \cap H^1(0,\tfinal;\Hs) \cap L^2(0,\tfinal;\Vs).
\end{equation}
\vspace{-0.4cm}
\begin{remark}
Since $\tfinal$ is finite and $L^{\infty}(0,\tfinal;\Vs)$ is continuously embedded into $L^2(0,\tfinal;\Vs)$, the fourth space in the definition of $\bigW$ is redundant. However, we keep it as written to emphasize that the $L^{\infty}(0,\tfinal;\Vs)$ and $L^2(0,\tfinal;\Vs)$ norms are utilized at different points in the analysis.
\end{remark}

Our notion of weak solution is as follows.  We say $\tQ \in \bigW$ is a weak solution of \cref{eq:forward_problem}, if $\tQ(0) \equiv \tQ(\cdot,t) |_{t=0} = \tQ_{0}$, and for a.e. $s \in (0,\tfinal)$, we have
\vspace{-0.2cm}
\begin{equation}\label{eq:weak_solution}
\begin{split}
\inner{\tQ_t(s)}{\tP(s)}{\Hs} + & \inner{\nabla \tQ(s)}{\nabla \tP(s)}{\Hs} + 
\frac{1}{\bulketa^2} \inner{\bulkfunc'(\tQ(s))}{\tP(s)}{\Hs} \\
+ \anchcoef \inner{\tQ(s)}{\tP(s)}{\HGs} & = \anchcoef \inner{\bdycon(s)}{\tP(s)}{\HGs} + \domcoef \inner{\domcon(s)}{\tP(s)}{\Hs},
\end{split}
\end{equation}
for all $\tP \in H^1(0,\tfinal;\Hs) \cap L^2(0,\tfinal;\Vs)$  with $\tP(0) = 0$, where we introduced the inner products on $\Hs$ and $\HGs$, respectively.

The solutions of the forward problem \cref{eq:weak_solution} are tensor-valued in space. There is little work on such problems in the control literature. Nevertheless, in many instances, we can exploit the Hilbert space structure on $\Vs$ or $\Hs$ and extend the derivations of typical energy estimates and Lipschitz continuity results. As a consequence,  the proofs of any results that follow the corresponding scalar or vector-valued cases without major changes have been drastically shortened and placed in the appendix.
%moved to the supplemental material.

%---------------------------------------------------------------
\subsection{Uniqueness of the state and a Lipschitz bound}\label{ssec:state_con}
%---------------------------------------------------------------

Under the assumption that solutions $\tQ$ with the appropriate regularity exist, we can prove Lipschitz continuity with respect to the input controls and therefore, a fortiori, uniqueness of solutions. Existence of solutions ultimately follows from a standard Galerkin approach.

\begin{theorem}[Continuous dependence on the data]\label{thm:soln_cont}
Let
$
\tQ_{1}, \tQ_{2} \in \bigW
% L^{\infty}(0,\tfinal;\Vs) \cap L^{\infty}(0,\tfinal;\HGs) \cap H^1(0,\tfinal;\Hs) \cap L^2(0,\tfinal;\Vs).
$
be two solutions of \cref{eq:weak_solution} corresponding to the input variables $\domconi{i}$, $\bdyconi{i}$, $\tQ_{0,i}$, for $i = 1,2$, which satisfy \cref{eq:minimal_input_reg}. Then there exists a constant $c > 0$, independent of the input variables, such that
\vspace{-0.2cm}
\begin{multline}\label{eq:continuity_bound}
\| \tQ_{1} - \tQ_{2} \|^2_{C([0,\tfinal];\Hs)} + 
\| \tQ_{1} - \tQ_{2} \|^2_{L^2(0,\tfinal;\Vs)} \\
\leq c 
\left(
\| \tQ_{0,1} - \tQ_{0,2} \|^2_{\Vs} + 
\domcoef \|\domconi{1} - \domconi{2}\|^2_{L^2(\domcyl)} +
\anchcoef \|\bdyconi{1} - \bdyconi{2}\|^2_{H^1(0,\tfinal;\HGs)}
\right).
\end{multline}
\end{theorem}

\begin{remark}
We use $c > 0$ as a generic constant throughout the text. We also note that the state space $\bigW$ can be compactly embedded into the space $C([0,\tfinal];\Hs)$.
%As a result, the use of the $C([0,\tfinal];\Hs)$-norm is justified in the Lipschitz bound.
\end{remark}

\begin{proof}
%See \cref{app:thm.3.2}.
The result follows by standard energy techniques, i.e. first test with the difference of solutions.   {\color{black}The lack of a monotone nonlinear operator is handled using the convex splitting \eqref{eqn:bulkfunc_convex_split}, which exploits the linear growth of $\psi'_e$. } Afterwards, we apply weighted Young's inequalities and the classic Gronwall lemma.  We omit the details.
\end{proof}

%\Cref{thm:soln_cont} also indicates the uniqueness of a solution, provided one exists. Indeed, if we set $(\tQ_{0,1},\domconi{1},\bdyconi{1}) = (\tQ_{0,2},\domconi{2},\bdyconi{2})$ then by \cref{eq:continuity_bound}, we deduce that $\tQ_{1} = \tQ_{2}$ almost everywhere on 
%$\domcyl$.

%---------------------------------------------------------------
\subsection{Existence and Energy Estimates}\label{ssec:fg_appr}
%---------------------------------------------------------------
This section is concerned with the existence of weak solutions. We use a Faedo-Galerkin approach, for which we require the following assumption. This condition will be tacitly assumed throughout the remainder of the text.

\begin{assumption}\label{as:findim}
For each $n \in \mathbb N$ (sufficiently large) there is an $n$-dimensional subspace $\Vs_n$  of $\Vs$ such that if $\{ \tY_{k} \}_{k=1}^n$ is the basis of $\Vs_{n}$ and $\Pi_n : \Vs \to \Vs_{n}$ is the linear projection onto $\Vs_{n}$, then $\Pi_n$ satisfies the following convergence property: $\| \Pi_{n} \tP - \tP \|_{\Vs} \to 0$, as $n \to \infty$, for all $\tP \in \Vs$.
%\begin{equation*}
%	\lim_{n \to \infty} \| \Pi_{n} \tP - \tP \|_{\Vs} = 0, \quad \text{for all } \tP \in \Vs.
%\end{equation*}
\end{assumption}

Typically $\Vs_n$ is based on the eigenvectors of the Laplacian associated to (in this case) homogeneous Robin boundary conditions. As another example, when $\Omega$ has a piecewise $C^2$ boundary, $\Vs_{n}$ can be a conforming finite element space with $n$ nodal degrees-of-freedom defined over a conforming (curvilinear) mesh of $\Om$.  
Next, we define  
$
\tQ^{n}_{0} := \Pi_n \tQ_{0}, 
$
which under \Cref{as:findim} means $\tQ^{n}_{0}$ converges strongly to $\tQ_{0}$ in $V$, as $n \to +\infty$.  Set $\tQ^n_{0,k} := (\tQ_{0},\tY_{k})_{\Hs} $ for each $k = 1,\dots,n$. 

%---------------------------------------------------------------
\subsubsection{Existence of a discrete solution in $H^1(0,\tfinal; \Vs_n)$}\label{ssec:exist_disc}
%---------------------------------------------------------------
We start with the existence of unique solutions to the semi-discrete system.
\begin{proposition}\label{prop:disc-soln}
There exists a unique solution $\tQ^{n} \in H^1(0,\tfinal; \Vs_n)$ such that
\vspace{-0.2cm}
\begin{equation}\label{eq:weak_discrete}
\aligned
\inner{\tQ^{n}_{t}(s)}{\tP}{\Hs} + \inner{\nabla \tQ^{n}(s)}{\nabla \tP}{\Hs} + \anchcoef \inner{\tQ^{n}(s)}{\tP}{\HGs} + 
\frac{1}{\bulketa^2} \int_{\Omega} \bulkfunc'(\tQ^{n}(s)) \dd \tP \\
= \domcoef \inner{\domcon(s)}{\tP}{\Hs} + \anchcoef \inner{\bdycon(s)}{\tP}{\HGs},
\endaligned
\end{equation}
for all  
% $\tP \in H^1(0,\tfinal;\Hs_n) \cap L^2(0,\tfinal;\Vs_n)$  with $\tP(0) = 0$
$\tP \in \Vs_{n}$ and for a.e.\ $s \in (0,\tfinal)$, 
 with $\tQ^{n}(0) = \tQ^{n}_{0}$.
\end{proposition}
\begin{proof}
The proof involves a standard application of the Carath\'eodory existence theorem, i.e. \cref{eq:weak_discrete} reduces to a system of coupled ODEs, for which existence and uniqueness is straightforward to show.
%and is included in \cref{app:prop.3.5}.
\end{proof}

%\subsubsection{First a priori estimate}\label{ssec:1st_apriori}
\subsubsection{A priori estimates}\label{ssec:apriori}
Given the existence of finite dimensional solutions, we now consider energy estimates.  Let $M_1 : \R^3 \to \R$ be given by 
$M_1(x_1,x_2,x_3) = c (x_1+ \eta_{\Omega} x_2 + \eta_{\Gm} x_3)$.  
For readability, we will often leave off the arguments of $M_1$ when it is clear in context; $c > 0$ is a generic constant that can be updated as needed. It will never depend on $n$, the controls, or the input data.
\begin{proposition}\label{prop:enrgi-a}
Suppose that $\tQ_{0}, \domcon, \bdycon$ satisfy \cref{eq:minimal_input_reg}. Then for all $n \in \mathbb N$,  the solutions $Q^n$ from \Cref{prop:disc-soln} satisfy the bound
\vspace{-0.2cm}
\begin{equation}\label{eq:first_apriori_5}
\| \tQ^{n}\|^2_{L^2(0,\tfinal;\Vs)} \le M_1(\|\tQ_{0}\|^2_{\Hs},\|\domcon\|^2_{L^2(\domcyl)}, \| \bdycon \|^2_{L^2(\bdycyl)}).
\end{equation}
\end{proposition}

\begin{proof}
The result follows by similar energy techniques as in the proof of
\cref{thm:soln_cont}; we omit the details.
%See \cref{app:prop.3.6}.
\end{proof}

%\subsubsection{Further a priori estimates and boundedness in $\mathcal{W}$}\label{sssec:sec_apriori}
By exploiting the Hilbert space structure and the nature of the finite dimensional inner products, we can again use standard derivation techniques to derive further energy estimates and ultimately prove that the sequence of finite dimensional solutions is bounded in $\mathcal{W}$.  To indicate the dependence of the bound on the controls, we define $M_2 : \R^4 \to \R$ by $M_2(x_1,x_2,x_3,x_4) = (1/2) \left(
\delta_1 x_1 + \delta_2 x_2 + \delta_3 x_3 + \delta_4 x_4 \right)$.  
The positive constants $\delta_1,\dots,\delta_4$ are arbitrary and can be adjusted as needed. Given input controls, we leave off the arguments and abbreviate both $M_1$ and $M_2$ by setting
\vspace{-0.2cm}
\[
\aligned
M_1 &\equiv 
M_1(\|\tQ_{0}\|^2_{\Hs},\|\domcon\|^2_{L^2(\domcyl)}, \| \bdycon \|^2_{L^2(\bdycyl)}),\\
M_2 &\equiv 
 M_2(
\| \bdycon(t) \|^2_{\HGs},
\| \bdycon(0)\|^2_{\HGs},
\|(\bdycon)_t \|^2_{L^2(\bdycyl_t)}, 
\| \domcon \|^2_{L^2(\domcyl_t)}
).
\endaligned
\]
\vspace{-0.3cm}
\begin{proposition}\label{prop:disc-w-estm}
Suppose that $\tQ_{0}, \domcon, \bdycon$ satisfy \cref{eq:minimal_input_reg}.  Then there exists an $M_0 \ge 0$ for all $n$ such that
\vspace{-0.2cm}
\begin{equation}\label{eq:M0-bound}
  \| \nabla \tQ^{n}_{0}\|^2_{\Hs} 
+
\frac{1}{\bulketa^2}\int_{\Omega} \bulkimp(\tQ^{n}_{0}) 
- \frac{1}{\bulketa^2}\int_{\Omega} \bulkexp(\tQ^{n}_{0}) 
+
\delta'_{2}\|\tQ^{n}_{0}\|^2_{\HGs}  
\le 
M_0,
\end{equation}
holds.  Up to rescaling by a generic constant, it also holds for all $n$ a.e.\ in $t$ that
\vspace{-0.2cm}
\begin{equation}\label{eq:sec_apriori_est_4-supp}
 \|\tQ^{n}_{t}\|^2_{L^2(\domcyl_{t})}
+ 
 \| \nabla \tQ^{n}(t)\|^2_{\Hs}
+
\|\tQ^{n}(t)\|^2_{\HGs}
\le
M_0 + M_1 + M_2.
%c'_y
%+  c'_u.
\end{equation}
Furthermore, up to rescaling by a generic constant, it holds for all $n$ a.e.\ in $t$ that
\vspace{-0.2cm}
\begin{equation}\label{eq:sec_apriori_est_6-supp}
\|\tQ^{n}_{t}\|^2_{L^2(\domcyl_{t})}
+ 
 \| \nabla \tQ^{n}(t)\|^2_{\Hs}
+
\|\tQ^{n}(t)\|^2_{\Hs}
\le  M_0 + M_1 + M_2.
\end{equation}
Finally, as a consequence of \cref{eq:sec_apriori_est_4-supp}, \cref{eq:sec_apriori_est_6-supp}, and \cref{eq:first_apriori_5}, 
the sequence of solutions $\left\{Q^n\right\}$, with $Q^n$ from \Cref{prop:disc-soln}, is bounded in $\mathcal{W}$.
\end{proposition}
\begin{proof}
%See the supplemental material \cref{sec:proof4}. 
See \cref{app:prop.3.7}.
\end{proof}

%\subsubsection{An improved Lipschitz estimate}\label{ssec:further_estm}
In order to obtain further properties of the control-to-state mapping, we need a stronger Lipschitz continuity result which we first state for the semi-discrete problem.
\begin{proposition}\label{prop:better-lip}
Suppose that $\tQ_{0,i}, \domconi{i}, \bdyconi{i}$, for $i = 1, 2$, satisfy \cref{eq:minimal_input_reg} and let $\tQ^{n}_{0,i} = \Pi \tQ_{0,i}$, where $\Pi$ is given in \cref{as:findim}.  
Then the corresponding solutions $Q^n_{i}$, for $i=1, 2$, satisfy the bound
\vspace{-0.3cm}
\begin{multline}\label{eq:local_Lip_W}
%\aligned
 \| \tQ^{n}_{1} - \tQ^{n}_{2} \|_{\bigW}
\leq\\ c \Big( 
\| \domconi{1} - \domconi{2} \|_{L^2(0,\tfinal;\Hs)}^2  +
 \| \bdyconi{1} - \bdyconi{2} \|_{H^1(0,\tfinal;\HGs)}^2  
+  \| \tQ^{n}_{0,1} -  \tQ^{n}_{0,2} \|^2_{\Vs}\Big)^{1/2},
\end{multline}
where $c > 0$ is a generic constant that does not depend on the controls, states, or $n$.
\end{proposition}

\begin{proof}
%See \cref{app:prop.3.8}.
The proof is straightforward and directly mirrors the proof of \cref{thm:soln_cont}.  In particular, the Lipschitz continuity of the gradient of the bulk energy term is essential.  We omit the details.
\end{proof}

\subsubsection{Passage to the limit}\label{sssec:limit}
In light of the uniform bounds and energy estimates on $\left\{\tQ^{n}\right\}$ provided above, we can now prove the existence of a solution. 
%The proof also follows the scalar case and can be found in the supplemental material.

\begin{theorem}\label{thm:exist_state}
For every triple $(\tQ_{0}, \domcon, \bdycon)$ that satisfies \cref{eq:minimal_input_reg},
%$(\domcon,\bdycon) \in L^2(\domcyl) \times H^1(0,\tfinal;\HGs)$, 
there exists a unique solution $\bar{\tQ} \in \bigW$ of the weak form \cref{eq:weak_solution} and a $\rho > 0$ such that
\vspace{-0.2cm}
\begin{equation}\label{eq:asym_5}
\| \bar{\tQ} \|_{\bigW} \le \rho.
\end{equation}
\end{theorem}
\begin{proof}
See \cref{app:thm.3.9}.
%See \cref{sec:proof6}.
\end{proof}

\begin{remark}\label{rem:unif_in_ctrls}
Note that despite the fact it depends on $\domcon$ and $\bdycon$, the constant $\rho$ in  \cref{eq:asym_5} can be bounded from above by a uniform constant, which is independent of $\domcon$ and $\bdycon$, provided the latter two are taken over a bounded set in the space $L^2(\domcyl) \times H^1(0,\tfinal;\HGs)$. This is a direct consequence of the structure of the a priori estimates.%; see supplemental material.
\end{remark}

The previous results also allow us to pass to the limit along a subsequence to obtain the following global Lipschitz bound from \cref{eq:local_Lip_W}:
\vspace{-0.2cm}
\begin{multline}\label{eq:lim_local_Lip_W}
 \| \tQ_{1} - \tQ_{2} \|_{\bigW} 
\leq\\ c \Big( 
\| \domconi{1} - \domconi{2} \|_{L^2(0,\tfinal;\Hs)}^2  +
 \| \bdyconi{1} - \bdyconi{2} \|_{H^1(0,\tfinal;\HGs)}^2  
+  \| \tQ_{0,1} -  \tQ_{0,2} \|^2_{\Vs}\Big)^{1/2},
\end{multline}
provided the controls are feasible.

The final result in this section involves the continuity of the states on the full space-time cylinder. In contrast to the results above, our nonlinear system in tensor-valued variables  inhibits a direct application of the standard techniques as can be found, e.g., in the relevant chapters in \cite{Troltzsch_book2010}. We require a few additional steps, which we provide here. The remainder can be found in the appendix.  %supplemental material.  
Note that the following argument is unique to the one elastic constant case. For more general elastic energy densities, we require new techniques to derive such continuity results in future studies.

To begin, given the existence of a solution $\bar{\tQ}$ in $\bigW$, we have $\bar{\tQ} \in L^{6}(\domcyl)$, which follows from the Sobolev embedding theorem and the fact that $\tfinal < +\infty$. Therefore, $\bar{\tQ}$ is the unique solution of the system of linear parabolic equations given by
\vspace{-0.3cm}
\begin{equation}\label{eq:forward_problem_bootstrap_1a}
\begin{split}
\widehat{\tQ}_t - \Delta \widehat{\tQ} 
&= \domcoef \domcon - \frac{1}{\bulketa^2} \psi'(\bar{\tQ}), \text{ in } \Om, \quad 
\partial_{\vnu} \widehat{\tQ} + \anchcoef \widehat{\tQ} = 
\anchcoef \bdycon,  \text{ on } \Gm,
\end{split}
\end{equation}
with $\widehat{\tQ}(\cdot,0) = \tQ_{0}$ in $\Om$.  
Next, we use the fact that  there exists a set of five symmetric, traceless, $3 \times 3$ orthonormal matrices $\{ \basis^{i} \}_{i=1}^5 \subset \R^{3 \times 3}$ such that every $\tQ \in \Vs$ admits the representation
$
\tQ = q_i \basis^{i}
$
where $q_i \in H^1(\Omega ; \R)$ for $i =1,\dots,5$. Let $\bar{q}_i$ denote the scalar-valued functions for the solution $\bar{\tQ}$. 

This decomposition was also exploited in \cite{Davis_SJNA1998}. It provides us with an isometric isomorphism between $\Vs$ and $H^1(\Omega ; \mathbb R^5)$ and allows us to split the tensor-based problem into five separate scalar parabolic equations with Robin boundary conditions.  In addition, we note that the second bound in \cref{eqn:mod_LdG_bulk_pot_bnds} implies that $\psi'(\bar\tQ)$ is also in $L^6(\domcyl)$.

Now that we can separate the system into independent scalar-valued equations, and apply the well-known regularity theory to obtain continuity of $\bar\tQ$. To be clear, we obtain continuity of each $\bar{q}_i$ via e.g., \cite[Thm. 5.5]{Troltzsch_book2010}, and consequently of $\bar{\tQ}$. The remainder of the proof is concerned with removing the dependency on $\bar{\tQ}$ from the upper bound. Since this does not require any special techniques for tensor-valued solutions, we have placed it in \cref{app:thm.3.11}.
% supplemental material.

\begin{theorem}\label{thm:higher_reg_thm}
If, in addition to \cref{eq:minimal_input_reg}, we have $\domcon \in L^{r}(\domcyl)$, $\bdycon \in L^{s}(\bdycyl)$, $\tQ_{0} \in C(\overline{\Omega})$, 
%\[
%\domcon \in L^{r}(\domcyl) \quad \bdycon \in L^{s}(\bdycyl) \quad \tQ_{0} \in C(\overline{\Omega})
%\]
with $r \in (5/2,6]$ and $s > 4$, then $\bar{\tQ} \in C(\overline{\domcyl})$.  
Next, let
\vspace{-0.1cm}
\begin{equation}\label{eq:def_UUg}
\conreg := [L^2(\domcyl) \times H^1(0,\tfinal;\HGs)] \cap [L^{r}(\domcyl) \times L^{s}(\bdycyl)],
\end{equation}
endowed with the natural norm
\vspace{-0.2cm}
\begin{equation}\label{eq:def_norm_UUg}
\| (\domcon,\bdycon) \|_{\conreg} = 
\max\{ 
\| (\domcon,\bdycon) \|_{L^2(\domcyl) \times H^1(0,\tfinal;\HGs)},
\| (\domcon,\bdycon) \|_{L^{r}(\domcyl) \times L^{s}(\bdycyl)}
\}.
\end{equation}
Then there exists a constant $c > 0$ independent of $\bar{\tQ}$, $(\domcon,\bdycon) \in \conreg$, and a constant $M_{0}$ such that
\vspace{-0.3cm} 
\begin{equation}\label{eq:unif_reg_bd}
 \| \bar{\tQ} \|_{C(\overline{\domcyl})} \le c\left(\sqrt{M_0} 
 + 
 \| (\domcon,\bdycon) \|_{\conreg}
+
\| \tQ_{0} \|_{C(\overline{\Omega})}
\right).
\end{equation}
\end{theorem}
\begin{remark}
If $\tQ_{0} \equiv 0$, then $M_0 = 0$. Moreover, since $c$ is independent of $(\domcon,\bdycon)$, we can vary on a ball in $\conreg$ and obtain a uniform bound on the solution operator $S(\domcon,\bdycon)$ in \cref{eq:solution_operator} below as a mapping from $\conreg$ into $C(\overline{\domcyl})$.
\end{remark}
\begin{proof}
See \cref{app:thm.3.11}.
%The full proof can be found in the supplemental material \cref{sec:proof7}.
\end{proof}

{\color{black}
\begin{remark}
In optimal control, proving the state variable is continuous on $\overline{\domcyl}$ is often useful for the derivation of optimality conditions for zero-order bound constraints, as it provides an essential constraint qualification in the convex setting. However, a constraint of the type ``$\tQ \ge 0$'' is not interesting for the current application. Nevertheless, the continuity of $\bar{\tQ}$ provides a justification for constraints of the type ``$|\bar{\tQ}| = 0$'' on lower-dimensional manifolds embedded in $\Omega$, which would correspond to the placement of defects. The analysis of this challenging type of constraint will be part of future research.
\end{remark}
}

{\color{black}
\begin{remark}
The energy estimates and related bounds derived in this section (and the associated appendices) can be useful for future work on the a priori numerical analysis of the optimal control problem as they would remain true if we replace the controls $\domcon$ and $\bdycon$ by, e.g., finite element approximations. At several points, we adjust the coefficients in $M_0$, $M_1$, $M_2$, the generic constant $c$, and $\rho$ in \eqref{eq:asym_5}  used throughout the text. Nevertheless, these arguments should be largely unaffected by the usage of discrete controls and ultimately  stable bounds.
\end{remark}
}
\section{Existence of optimal controls}\label{sec:exist_ctrl}

We denote the control-to-state operator for the forward problem \cref{eq:weak_solution} by
\vspace{-0.2cm}
\begin{equation}\label{eq:solution_operator}
	S : L^2(\domcyl) \times H^1(0,\tfinal;\HGs) \to \bigW,
\end{equation}
i.e. $S(\domcon,\bdycon) \in \bigW$ solves \cref{eq:weak_solution} for any controls $(\domcon,\bdycon) \in L^2(\domcyl) \times H^1(0,\tfinal;\HGs)$.

\begin{theorem}\label{thm:exst_ctrl}
Suppose that the set of control constraints  
$\conadmis$ is a nonempty, closed, and convex subset of  $L^2(\domcyl) \times H^1(0,\tfinal;\HGs)$. Then the optimal control problem \cref{eq:optconprob}-\cref{eq:con_constr} admits a solution.
\end{theorem}

\begin{remark}
The assumption that $\conadmis \cap [L^2(\domcyl) \times H^1(0,\tfinal;\HGs)]$ is closed can be guaranteed in a variety of contexts, e.g. for pointwise a.e.\ bound constraints. Moreover, if the boundary control is independent of time or only applied at a \emph{finite} number of points in time, as in many applications, then this assumption is fulfilled.
\end{remark}

\begin{proof}
For readability, we set 
$\tilconadmis := \conadmis \cap [L^2(\domcyl) \times H^1(0,\tfinal;\HGs)]$.  
By assumption, this set is nonempty, closed, and convex and therefore weakly closed in $L^2(\domcyl) \times H^1(0,\tfinal;\HGs)$.  
In addition, we restrict the control-to-state operator $S(\domcon,\bdycon)$ to $\tilconadmis$.

By hypothesis, $\tilconadmis  \neq \emptyset$. Consequently, there exists a minimizing sequence $\left\{(\domconi{n},\bdyconi{n})\right\} \subset\tilconadmis $ for \cref{eq:optconprob}-\cref{eq:con_constr}.  Clearly,  $\left\{(\domconi{n},\bdyconi{n})\right\}$ is uniformly bounded in $L^{2}(\domcyl) \times L^{2}(\bdycyl)$. Moreover, since $\left\{(\domconi{n},\bdyconi{n})\right\}$ is a minimizing sequence, there exists an $n_0 \in \mathbb N$ such that for all $n \ge n_0$:
\vspace{-0.3cm}
\begin{multline*}
\left\{(\domconi{n},\bdyconi{n})\right\} \subset \\
\left\{ 
(\domcon,\bdycon) \in \conadmis 
\left| 
\costobj(S(\domcon,\bdycon),\domcon,\bdycon) 
\le 
\costobj(S(\domconi{n_0},\bdyconi{n_0}),\domconi{n_0},\bdyconi{n_0})\right.\right\}.
\end{multline*}
By the definition of $\costobj$, there is a constant $c > 0$, such that $\| \bdyconi{n} \|_{H^1(0,\tfinal;\HGs)} \le c$, for all $n \ge 1$.  
It follows that there exists a subsequence  $\left\{(\domconi{n_{k}},\bdyconi{n_{k}})\right\} $ and (weak) limit point $(\bar{\domcon},\barbdycon) \in \tilconadmis$. 
%More precisely, we have $\domcon_{n_k} \to \bar{\domcon}$ weakly-$*$ in $L^{\infty}(\domcyl)$ and $\bdyconi{n_k} \to \barbdycon$ weakly-$*$ in $L^{\infty}(\bdycyl)$ and weakly in $H^1(0,\tfinal;\HGs)$.
Finally, it follows in light of \Cref{thm:exist_state}, \Cref{rem:unif_in_ctrls}, equation \cref{eq:asym_5}, and the Aubin-Lions-Lemma that there exists a subsequence $\{\tQ^{l}\}$ with $\tQ^{l} := S(\domconi{n_{k_{l}}},\bdyconi{n_{k_{l}}})$ that converges to $\bar{\tQ}$ such that
\begin{itemize}
\item $\tQ^{l} \to \bar{\tQ}$ weakly$^*$ in $L^{\infty}(0,\tfinal;\Vs)$,
\item $\tQ^{l} \to \bar{\tQ}$ weakly in $L^{2}(0,\tfinal;\Vs)$,
\item $\tQ^{l} \to \bar{\tQ}$ weakly$^*$ in $L^{\infty}(0,\tfinal;\HGs)$,
\item $\tQ^{l} \to \bar{\tQ}$ weakly in $H^1(0,\tfinal;\Hs)$,
\item $\tQ^{l} \to \bar{\tQ}$ strongly in $C([0,\tfinal];L^{6-\varepsilon}(\Omega,\symmtraceless))$ for $\varepsilon \in (0,5]$.
\end{itemize}
Using analogous arguments to those in \cref{app:thm.3.9},
%\cref{sssec:limit}, 
it follows that $\bar{\tQ} = S(\bar{\domcon},\barbdycon)$. Finally, the weak lower-semicontinuity of $\costobj$ along with the properties of $\tQ^{l}$ and $(\domconi{n_{k_{l}}},\bdyconi{n_{k_{l}}})$ guarantee that $(\bar{\domcon},\barbdycon)$ is an optimal solution of \cref{eq:optconprob}-\cref{eq:con_constr}. 
\end{proof}

\section{First-Order Optimality Conditions and the Adjoint Equation}\label{sec:first_order}
%\subsection{Differential Sensitivity of $S$}
We first derive a differentiability result for $S$.
\begin{theorem}\label{thm:deriv_cts}
%In addition to \cref{eq:minimal_input_reg}, suppose that $\tQ_{0} \in \Vs \cap C(\overline{\Omega})$. 
Under the hypotheses of Theorem \ref{thm:higher_reg_thm},
the control-to-state mapping $S$ from $\conreg$ into $\bigW$ is Fr\'echet differentiable.  Moreover, given $(\domcon,\bdycon)$, $(\tH_{\Om},\tH_{\Gm})$ in $\conreg$, the derivative of $S$ at $(\domcon,\bdycon)$ in direction $(\tH_{\Om},\tH_{\Gm})$ is given by the unique solution 
$
\tXi = S'_{\domcon,\bdycon}(\tH_{\Om},\tH_{\Gm})
$ of \cref{eq:forward_problem_deriv}.
\end{theorem}
\begin{proof}
%We begin with the following observation. 
Let $(\domcon,\bdycon),(\tH_{\Om},\tH_{\Gm}) \in L^2(\domcyl) \times L^2(\bdycyl)$, denote $\bar{\tQ} := S(\domcon,\bdycon)$, $\tQ_{\tH} := S(\domcon+\tH_{\Om},\bdycon + \tH_{\Gm})$, and let $\tXi$ be the solution to
\vspace{-0.2cm}
\begin{equation}\label{eq:forward_problem_deriv}
\begin{split}
\tXi_t - \Delta \tXi + \frac{1}{\bulketa^2} \bulkfunc''(\bar{\tQ}) \dd \tXi
&= \domcoef \tH_{\Om}, \text{ in } \Omega, \quad 
\partial_{\vnu} \tXi + \anchcoef \tXi
= \anchcoef \tH_{\Gm}, \text{ on } \Gm,
\end{split}
\end{equation}
with $\tXi(\cdot,0) = 0$.  
%In order for this to be well-defined, we also require $(\domcon,\bdycon) \in L^r(\domcyl) \times L^s(\bdycyl)$ with $r > 5/2$ and $s > 4$ and $\tQ_{0} \in C(\overline{\Omega})$; as
It follows from \Cref{thm:higher_reg_thm} that 
%this implies 
$\bulkfunc''(\bar{\tQ})$ is continuous on $\overline{\domcyl}$. 
%via \Cref{thm:higher_reg_thm}.  
Next, we set 
$
\tA_{\tH} := \tQ_{\tH} - \bar{\tQ} - \tXi
$
and show that it behaves like $o(\|(H_{\Omega},H_{\Gamma})\|)$ in the appropriate norms.  
Almost everywhere on $\domcyl$, we have
\vspace{-0.2cm}
\begin{equation*}
\begin{split}
    \bulkfunc'(\tQ_{\tH}) - \bulkfunc'(\bar{\tQ}) - \bulkfunc''(\bar{\tQ}) \dd \tXi =
    \bulkfunc''(\bar{\tQ}) \dd \tA_{\tH} - \tX_{\tH},
\end{split}
\end{equation*}
where $\tX_{\tH} := -\int_{0}^{1} \bulkfunc''(\bar{\tQ} + \tau(\tQ_{\tH} - \bar{\tQ}))- \bulkfunc''(\bar{\tQ}) d\tau \dd [\tQ_{\tH} - \bar{\tQ}]$.  
With the extended assumptions of \Cref{thm:exist_state}, $\tA_{\tH}$ satisfies
\vspace{-0.2cm}
\begin{equation}\label{eq:linear_forward_h_dep}
\begin{split}
(\tA_{\tH})_t - \Delta \tA_{\tH} + \frac{1}{\bulketa^2}\bulkfunc''(\bar{\tQ}) \dd \tA_{\tH}
&= \frac{1}{\bulketa^2} \tX_{\tH}, \text{ in } \Omega, ~
\partial_{\vnu} \tA_{\tH} + \anchcoef \tA_{\tH}
= 0, \text{ on } \Gm,
\end{split}
\end{equation}
with $\tA_{\tH}(\cdot,0) = 0$.  The system \cref{eq:linear_forward_h_dep} is a simplified version of the nonlinear forward problem. Therefore, using a slight modification of the same arguments, we can prove that $\tA_{\tH} \in \bigW$.  In particular, we readily obtain the following bound (for a generic constant $c > 0$ independent of $\tH = (\tH_{\Om},\tH_{\Gm})$): 
$
\| \tA_{\tH} \|_{\bigW} \le c \| (\tX_{\tH},0) \|_{L^2(\domcyl) \times L^2(\bdycyl)}.
$  
%In order to obtain the desired differentiability statement, we need to more closely investigate the $L^2(\domcyl)$-norm of $\tX_{\tH}$.

As a consequence of the Lipschitz continuity of $\psi''$, we have (a.e.\ on $\domcyl$)
\vspace{-0.3cm}
\[
| \tX_{\tH} | \le 
\int_{0}^{1} |\bulkfunc''(\bar{\tQ} + \tau(\tQ_{\tH} - \bar{\tQ}))- \bulkfunc''(\bar{\tQ})| d\tau |\tQ_{\tH} - \bar{\tQ}|,
\le c |\tQ_{\tH} - \bar{\tQ}|^2,
\]
where $c$ represents the Lipschitz modulus for $\psi''$. Note that $\psi'''(Q)$ is only nonzero on the set where $\tr(Q^2) = |Q|^2$ is between two fixed constants $b_1$ and $b_2$, i.e., where $Q$ is bounded in space-time.  Then for a.e.\ $t \in (0,\tfinal)$ we obtain the bound: 
$
\aligned
\| \tX_{\tH}(t) \|^2_{\Hs} 
 \le c\|\tQ_{\tH}(t) - \bar{\tQ}(t)\|^4_{L^4(\Omega)}. 
\endaligned
$
It follows from
\cref{eq:lim_local_Lip_W}, the definition of $\bigW$, and the Sobolev embedding theorem that 
$
 \|\tQ_{\tH}(t) - \bar{\tQ}(t) \|_{L^p(\Omega)} 
\leq
  c\| (\tH_{\Om},\tH_{\Gm}) \|_{\conreg}, 
$
for a.e.\ $t \in (0,\tfinal)$ and any $p \in [1,6]$.  Consequently, we have
\vspace{-0.2cm}
\[
\aligned
\| \tX_{\tH}(t) \|^2_{\Hs} 
&\le c\| (\tH_{\Om},\tH_{\Gm}) \|_{\conreg}^2 \|\tQ_{\tH}(t) - \bar{\tQ}(t)\|^2_{L^4(\Omega)}.
\endaligned
\]
Integrating in time and taking the square root, we obtain
\vspace{-0.2cm}
\[
\aligned
\| \tX_{\tH} \|_{L^2(\bdycyl)} 
&\le c\| (\tH_{\Om},\tH_{\Gm}) \|_{\conreg}
\|\tQ_{\tH} - \bar{\tQ}\|_{L^2(0,\tfinal;L^4(\Omega))},
\endaligned
\]
which immediately yields $\| (\tX_{\tH},0) \|_{L^2(\domcyl) \times L^2(\bdycyl)} = o(\| (\tH_{\Om},\tH_{\Gm}) \|_{\conreg})$.
%since $\|\tQ_{\tH} - \bar{\tQ}\|_{L^2(0,\tfinal;L^4(\Omega))} = O(\| (\tH_{\Om},\tH_{\Gm}) \|_{\conreg})$.
\end{proof}
%We have then proven the following theorem.
%\begin{theorem}\label{thm:deriv_cts}
%In addition to \cref{eq:minimal_input_reg}, suppose that $\tQ_{0} \in \Vs \cap C(\overline{\Omega})$. Then the control-to-state mapping $S$ from $\conreg$ into $\bigW$ is Fr\'echet differentiable.  Moreover, given $(\domcon,\bdycon),(\tH_{\Om},\tH_{\Gm}) \in \conreg$, the derivative of $S$ at $(\domcon,\bdycon)$ in direction $(\tH_{\Om},\tH_{\Gm})$ is given by the unique solution 
%$
%\tXi = S'_{\domcon,\bdycon}(\tH_{\Om},\tH_{\Gm})
%$ of \cref{eq:forward_problem_deriv}.
%\end{theorem}

%This leads to the following corollary.
\begin{corollary}\label{cor:red_grad}
Under the assumptions of \Cref{thm:deriv_cts}, the reduced objective function $\reducedobj : \conreg \to \overline{\R}$ defined by
\vspace{-0.2cm}
\begin{equation}\label{eq:reduced_obj}
\aligned
\reducedobj(\domcon,\bdycon) 
:=
&\frac{\beta_{\domcyl}}{2} \| S(\domcon,\bdycon) - \domtarget \|^2_{L^2(\domcyl)} + 
\frac{\beta_{\bdycyl}}{2} \| S(\domcon,\bdycon)|_{\Gm} - \bdytarget \|^2_{L^2(\bdycyl)} \\
+ \frac{\beta_{\tfinal}}{2} & \| S(\domcon,\bdycon)(\cdot,T) - \fintarget \|^2_{\Hs} + 
\frac{\alpha_{\domcyl}}{2} \| \domcon \|^2_{L^2(\domcyl)} + 
\frac{\alpha_{\bdycyl}}{2} \| \bdycon \|^2_{H^1(0,\tfinal;\HGs)},
\endaligned
\end{equation}
is Fr\'echet differentiable.  Furthermore, given $\bar{\tQ} = S(\bardomcon, \barbdycon)$, a direction $(\tH_{\Om},\tH_{\Gm}) \in \conreg$, and $\tXi = S'_{\bardomcon,\barbdycon}(\tH_{\Om},\tH_{\Gm})$, the associated solution of \cref{eq:forward_problem_deriv}, the directional derivative of $\reducedobj$ at $(\bardomcon,\barbdycon)$ in direction $(\tH_{\Om},\tH_{\Gm})$ is given by:
\vspace{-0.2cm}
\begin{equation}\label{eq:dir_der}
\aligned
\reducedobj_{\bardomcon,\barbdycon}' &(\tH_{\Om},\tH_{\Gm}) 
=  
\beta_{\domcyl} \inner{\bar{\tQ} - \domtarget}{\tXi}{L^2(\domcyl)} +
\beta_{\bdycyl} \inner{\bar{\tQ}_{\Gm} -  \bdytarget}{\tXi}{L^2(\bdycyl)} \\
+ \beta_{\tfinal} & \inner{\bar{\tQ}(\tfinal) - \fintarget}{\tXi(\tfinal)}{\Hs} + 
\alpha_{\domcyl} \inner{\bardomcon}{\tH_{\Om}}{L^2(\domcyl)} +
\alpha_{\bdycyl} \inner{\barbdycon}{\tH_{\Gm}}{H^1(0,\tfinal;\HGs)}.
\endaligned
\end{equation}
Moreover, if $\tR$ is the unique weak solution of the linear parabolic (adjoint) equation:
\vspace{-0.2cm}
\begin{subequations}\label{eq:adjoint_equation}
\begin{align}
-\tR_t - \Delta \tR + \frac{1}{\bulketa^2} \bulkfunc''(\bar{\tQ}) \dd \tR
&=  \beta_{\domcyl}(\bar{\tQ} - \domtarget),\qquad~~ \text{ in } \Omega \times (0,\tfinal) \\
\partial_{\vnu} \tR + \anchcoef \tR
&= 
\beta_{\bdycyl}(\bar{\tQ}_{\Gm} -  \bdytarget), \qquad \text{ on } \Gm \times (0,\tfinal) \\
\tR(\cdot,\tfinal) &= \beta_{\tfinal}(\bar{\tQ}(\tfinal) - \fintarget), \quad \text{ in } \Omega,
\end{align}
\end{subequations}
then we have
\vspace{-0.2cm}
\begin{equation}\label{eq:red_grad}
\begin{split}
\reducedobj'_{\bardomcon,\barbdycon}(\tH_{\Om},\tH_{\Gm}) &= \domcoef \inner{\tR}{\tH_{\Om}}{L^2(\domcyl)}
 + 
\anchcoef \inner{\tR}{\tH_{\Gm}}{L^2(\bdycyl)} \\
& \qquad \quad + \alpha_{\domcyl} \inner{\bardomcon}{\tH_{\Om}}{L^2(\domcyl)} +
\alpha_{\bdycyl} \inner{\barbdycon}{\tH_{\Gm}}{H^1(0,\tfinal;\HGs)}.
\end{split}
\end{equation}
% the gradient of $\reducedobj(\bardomcon,\barbdycon)$, denoted by $\nabla \reducedobj(\bardomcon,\barbdycon)$, is fully characterized by the 
\end{corollary}
\begin{proof}
The differentiability of the reduced objective functional is a consequence of \Cref{thm:deriv_cts}, the smoothness of the original tracking-type functional, and the chain rule.  This yields \cref{eq:dir_der}. For the equivalent characterization \cref{eq:red_grad}, we use   \cref{eq:dir_der} and the adjoint equations \cref{eq:adjoint_equation} by following the standard computations for the adjoint calculus, see e.g., \cite{Troltzsch_book2010}.
\end{proof}

%\subsection{First-Order Optimality Conditions}
\Cref{thm:deriv_cts} and \Cref{cor:red_grad} provide us with first-order optimality conditions of primal and dual type and a means of efficiently calculating derivatives of the reduced objective functional, which are needed for numerical methods. 
%We state the primal conditions here and derive the dual conditions along with the adjoint equation in the next subsection. 

\begin{theorem}
In addition to \cref{eq:minimal_input_reg}, suppose $\tQ_{0} \in \Vs \cap C(\overline{\Omega})$
and 
$
\conadmis \cap \conreg \ne \emptyset.
$
If the optimal solution  $(\bardomcon,\barbdycon)$ of  \cref{eq:optconprob}-\cref{eq:con_constr} is in $\conadmis \cap \conreg$, then the following variational inequality holds:
\vspace{-0.2cm}
\begin{multline}\label{eq:var_ineq}
\domcoef \inner{\tR}{\domcon - \bardomcon}{L^2(\bdycyl)} + 
\anchcoef \inner{\tR}{\bdycon - \barbdycon}{L^2(\bdycyl)} \\
+ \alpha_{\domcyl} \inner{\bardomcon}{\domcon - \bardomcon}{L^2(\domcyl)} +
\alpha_{\bdycyl} \inner{\barbdycon}{\bdycon - \barbdycon}{H^1(0,\tfinal;\HGs)} \ge 0,
\end{multline}
for all $(\domcon, \bdycon ) \in \conadmis \cap \conreg$, where $\tR$ solves \cref{eq:adjoint_equation} with $\bar{\tQ} = S(\bardomcon,\barbdycon)$.
\end{theorem}
\begin{proof}
This is an immediate consequence of \Cref{thm:deriv_cts} and \Cref{cor:red_grad}. To see this, note that 
$
\reducedobj(\bardomcon,\barbdycon) \le 
\reducedobj(\domcon, \bdycon) \quad \forall (\domcon, \bdycon ) \in \conadmis
$.  
By assumption, $\conadmis \cap \conreg$ is a nonempty convex set. Therefore, the previous relation gives us the difference quotients
\[
0 \le \lambda^{-1}(\reducedobj(\bardomcon + \lambda \tH_{\Om},\barbdycon + \lambda \tH_{\Gm}) - \reducedobj(\bardomcon ,\barbdycon)),
\]
where $\lambda \in (0,1)$ and $(\tH_{\Om},\tH_{\Gm}) = (\domcon, \bdycon ) - (\bardomcon, \barbdycon)$ with $(\domcon,\bdycon ) \in \conadmis \cap \conreg$. The rest follows from \Cref{thm:deriv_cts} and \Cref{cor:red_grad}; in particular \cref{eq:red_grad}.
\end{proof}

%---------------------------------------------------------------
\section{Finite Element Approximation}\label{sec:FEM_approx}
%---------------------------------------------------------------

We discretize \cref{eq:weak_solution} in the following way.  First, we assume that $\Om$ is polyhedral so that it can be represented exactly by a conforming triangulation $\Tk_{h} = \{ T_{i} \}$ of shape regular simplices (e.g. tetrahedra), where $h = \max_{T \in \Tk_{h}} \diam(T)$.  In other words, $\Om \equiv \cup_{T \in \Tk_{h}} T$.  Curved domains can also be considered; the polyhedral assumption is only for simplicity.

Next, we define the space of continuous piecewise polynomial functions on $\Om$: $\M_{h}^{k} (\Om) := \left\{ v \in C^{0} (\Om) \mid v |_{T} \in \Pk_{k} (T), ~\forall T \in \Tk_{h} \right\}$, for $k \geq 1$, and we reserve $\M_{h}^{0} (\Om)$ for piecewise constant functions.  
%\begin{equation}\label{eqn:std_Pk_FE_space}
%\begin{split}
%\M_{h}^{k} (\Om) := \left\{ v \in C^{0} (\Om) \mid v |_{T} \in \Pk_{k} (T), ~\forall T \in \Tk_{h} \right\},
%\end{split}
%\end{equation}
%where $\Pk_{k} (T)$ is the space of polynomials of degree $\leq k$ on $T$, for $k \geq 1$.  When $k=0$, we have {shorten}
%\begin{equation}\label{eqn:std_P0_FE_space}
%\begin{split}
%\M_{h}^{0} (\Om) := \left\{ v \in L^2 (\Om) \mid v |_{T} \in \Pk_{0} (T), ~\forall T \in \Tk_{h} \right\},
%\end{split}
%\end{equation}
%which consists of piecewise constant functions.
Let $\{ \basis^{i} \}_{i=1}^5$ be a basis of $\symmtraceless$.  We then define the following continuous, piecewise linear approximation of $\Vs$:
\vspace{-0.3cm}
\begin{equation}\label{eqn:Q-tensor_FE_space}
\begin{split}
	\Vs_{h} := \left\{ \tP \in C^{0} (\Om;\symmtraceless) \mid \tP = \sum_{i=1}^5 p_{i,h} \basis^{i}, ~ p_{i,h} \in \M_{h}^{1} (\Om), ~ 1 \leq i \leq 5 \right\} \subset \Vs,
\end{split}
\end{equation}
and denote by $\interp$ the standard Lagrange interpolant on $\Vs_{h}$.  Therefore, we approximate $\tQ \in L^2(0,\tfinal;\Vs)$ by $\tQ_{h} \in H^1(0,\tfinal;\Vs_{h})$, i.e. piecewise linear in time; c.f. \cite{Davis_SJNA1998}.  We also introduce the following piecewise constant approximations of $\Hs$ and $\HGs$, respectively, for approximating the controls $\domcon$, $\bdycon$:
\vspace{-0.3cm}
\begin{equation}\label{eqn:control_FE_space}
\begin{split}
	\Hs_{h} &:= \left\{ \tP \in L^2 (\Om;\symmtraceless) \mid \tP = \sum_{i=1}^5 p_{i,h} \basis^{i}, ~ p_{i,h} \in \M_{h}^{0} (\Om), ~ 1 \leq i \leq 5 \right\} \subset \Hs, \\
	\HGsh &:= \left\{ \tP \in L^2 (\Gm;\symmtraceless) \mid \tP = \sum_{i=1}^5 p_{i,h} \basis^{i}, ~ p_{i,h} \in \M_{h}^{0} (\Gm), ~ 1 \leq i \leq 5 \right\} \subset \HGs,
\end{split}
\end{equation}
where $\M_{h}^{0} (\Gm) := \left\{ v \in L^2 (\Gm) \mid v |_{F} \in \Pk_{0} (F), ~\forall F \in \Fk_{h} \right\}$, where $\Fk_{h} = \{ F \}$ is the set of faces that make up $\partial \Om$.  Thus, we approximate $\domcon \in L^2(0,\tfinal;\Hs)$, $\bdycon \in H^1(0,\tfinal;\HGs)$ by $\domconh \in H^1(0,\tfinal;\Hs_{h})$, $\bdyconh \in H^1(0,\tfinal;\HGsh)$, respectively.  The control bounds are enforced at the nodal degrees-of-freedom of $\Hs_{h}$ and $\HGsh$, i.e. at the centroid of the mesh elements.

Furthermore, we discretize the time interval $[0, \tfinal]$ into a union of $K$ sub-intervals of uniform length $\dt$, i.e. time-steps.  With this, we write $\tQ_{h}^{k}(x) := \tQ_{h}(x,k \dt)$, and approximate $\tQ_{t}(x,k \dt)$ by the finite difference quotient: $\dt^{-1} \left( \tQ_{h}^{k+1}(x) - \tQ_{h}^{k}(x) \right)$.  
%\begin{equation*}
%	\tQ_{t}(x,k \dt) \approx \frac{\tQ_{h}^{k+1}(x) - \tQ_{h}^{k}(x)}{\dt}.
%\end{equation*}
In addition, the time-dependence of the controls is written $\domconh^{k}(x) := \domconh(x,k \dt)$, $\bdyconh^{k}(x) := \bdyconh(x,k \dt)$.

The fully discrete version of \cref{eq:weak_solution} is as follows.  Given the initial condition $\tQ^{0}(\cdot) := \interp \tQ(\cdot,t=0)$, and controls $\{ \domconh^{k} \}_{k=0}^{K} \subset \Hs_{h}$, $\{ \bdyconh^{k} \}_{k=0}^{K} \in \HGsh$, we iteratively solve the following implicit equation for $k = 0, ..., K-1$: find $\tQ_{h}^{k+1} \in \Vs_{h}$ such that
\begin{equation}\label{eq:FEM_var_form}
\begin{split}
\dt^{-1} \inner{\tQ_{h}^{k+1} - \tQ_{h}^{k}}{\tP_{h}}{\Hs} + & \inner{\nabla \tQ_{h}^{k+1}}{\nabla \tP_{h}}{\Hs} + 
\frac{1}{\bulketa^2} \inner{\bulkfunc'(\tQ_{h}^{k+1})}{\tP_{h}}{\Hs} \\
+ \anchcoef \inner{\tQ_{h}^{k+1}}{\tP_{h}}{\HGs} = \anchcoef & \inner{\bdyconh^{k+1}}{\tP_{h}}{\HGs} + \domcoef \inner{\domconh^{k+1}}{\tP_{h}}{\Hs}, ~~ \forall ~ \tP_{h} \in \Vs_{h}.
\end{split}
\end{equation}
For $\dt$ sufficiently small, depending on $\bulketa$, \cref{eq:FEM_var_form} is monotone at each time-step and can be effectively solved with Newton's method.  
Similar to \cref{eq:forward_problem}, \cref{eq:FEM_var_form} is a tensor-valued version of a discrete Allen-Cahn equation with Robin boundary conditions.  Convergence of $\tQ_{h}$ to the exact solution $\tQ$ of \cref{eq:forward_problem} follows from the standard theory for semi-linear parabolic problems; c.f. \cite{Davis_SJNA1998,Shen_DCDS2010,Zhao_JCP2016,Zhao_JSC2016}.  The adjoint problem is solved in an analogous way, using a similar discretization; since the adjoint PDE \cref{eq:adjoint_equation} is linear (variable coefficient), Newton's method is not required.

%%%%%%%%%%%%%%%%%%%%%%%%%%%%%%%%%%%%%%%%%%%%%%%%%%%%%%%%%%%%%%%%%%%%%%
\section{Numerical Results}\label{sec:results}
%%%%%%%%%%%%%%%%%%%%%%%%%%%%%%%%%%%%%%%%%%%%%%%%%%%%%%%%%%%%%%%%%%%%%%

We approximate minimizers of \cref{eq:optconprob}, by discretizing the forward problem in \cref{eq:weak_solution} with the finite element method described in \cref{sec:FEM_approx}.  Moreover, the time integrals present in \cref{eq:optconprob} are discretized with the trapezoidal rule.  This leads to a discrete form of the adjoint problem in \cref{eq:adjoint_equation} along with the corresponding discrete form of the derivative functional \cref{eq:red_grad}.  Thus, we use a projected gradient optimization method, with a back-tracking line-search, see e.g., \cite{Bertsekas_IEEETAC1976,Dunn_SJCO1981}, to compute (discrete) optimal solutions of \cref{eq:optconprob}.  During the line-search, we compute the projection onto the convex set in \cref{eq:con_constr} by straightforward normalization of the current guess for $\domcon$ and $\bdycon$.  The entire algorithm was implemented in NGSolve \cite{Schoeberl_TR2014}.

We present examples when the dimension $d$ is $2$ or $3$.  Our experiments involve tensor quantities that are uniaxial (recall \cref{eqn:uniaxial} when $d=3$).  For any $d > 1$, a uniaxial $\tQ$ has the form $\tQ_{ij} = s_{*} \left( n_i n_j - \delta_{ij}/d \right)$, for $1 \leq i,j \leq d$, where $[n_{i}]_{i=1}^d \equiv n \in \R^d$ is a unit vector (often called the \emph{director}) and $s_{*}$ depends on the coefficients in $\bulkfunc$.

{The concept of \emph{defect} is ubiquitous in liquid crystals and plays a critical role in our numerical experiments.  Assuming that $\tQ$ has a uniaxial form, a defect corresponds to a discontinuity in the director $n$.  Let $d=2$ and suppose $n : \R^2 \to \R^2$ is a vector field defined in the plane, continuous everywhere except at isolated points.  The \emph{index} of $n$, about a point of discontinuity $p_0$, is simply the number of full rotations of $n(a(s))$ along a closed path $a(s)$ that surrounds $p_0$ (see \cite[pg. 280]{doCarmo:book}).  For vector fields, the index is always an integer.  If $\hat{n} : \R^2 \to \mathbb{RP}^{1}$, i.e. $\hat{n}(x) \equiv -\hat{n}(x)$ (also known as a line field), then the index may be a half-integer.  One can represent a line field $\hat{n}$ with a vector field $n$ (see \cite{Ball_PAMM2007}), and vice-versa, in the sense that $\hat{n} \otimes \hat{n} = n \otimes n \equiv [n_i n_j]_{i,j=1}^{d}$.  Thus, since all $\tQ$-tensors are uniaxial in dimension $d=2$, and the algebraic form of a uniaxial $\tQ$ involves $n \otimes n$, the \emph{degree} of the defect of $\tQ$ at $p_0$ is simply the index of the director $n$ (or equivalently $\hat{n}$) about $p_0$.}

{In dimension $d=3$, the degree of a defect makes sense relative to a plane in $\R^3$.  For instance, if the set of defects (points of discontinuity) forms a $C^1$ curve, $\alpha$, in $\R^3$, then the degree of a point on that curve is computed relative to the normal plane of the curve.  In other words, let $n$ be the eigenvector of $\tQ$ with largest eigenvalue and define the degree of the defect to be the index of $n$ with respect to a closed curve (in the normal plane) around a point of discontinuity in $\alpha$.}

{However, the LdG model will not create point (line) discontinuities in dimension $d=2$ ($d=3$) because $\tQ(t,\cdot) \in H^1(\Om;\symmtraceless)$.  Therefore, any potential discontinuities get smoothed out causing $\tQ$ to vanish there (i.e. the liquid crystal ``melts'').  Thus, in the LdG model, the location of defects are usually identified with regions where $\tQ_{ij}=0$ for $1 \leq i,j \leq d$.  For more information on defects, see \cite{Brinkman_PT1982,Brezis_CMP1986,Schopohl_PRL1987,Virga_book1994,Gu_PRL2000,LinLiu_JPDE2001,Musevic_Sci2006,Kralj_PRSA2014,BorthagarayWalker_chap2021}.}

\subsection{Control of a $+1/2$ degree point defect in two dimensions}\label{sec:ctrl_+1/2_defect_2D}

The domain is the unit square $\Om = (0,1)^2$ and the parameters of the forward problem are as follows.  The coefficients of the double well in \cref{eqn:Landau-deGennes_bulk_potential} are
\vspace{-0.2cm}
\begin{equation}\label{eqn:+1/2_defect_2D_bulk_pot_coefs}
\begin{split}
	\bulkK = 1, \quad \bulkA = 16.32653061225, \quad \bulkB = 0, \quad \bulkC = 66.63890045814,
\end{split}
\end{equation}
and $\bulkfunc(\tQ)$ has a global minimum at $\tQ_{*} = s_{*} \left[ n_i n_j - \delta_{ij}/2 \right]_{i,j=1}^{2}$, where $n \in \R^2$ is any unit vector, and $s_{*} = 0.7$.  The other coefficients are given by $\bulketa = 0.2$, $\domcoef = 0$, $\anchcoef = 100$.

The initial condition was defined as follows.  First, let $n = n(x_1,x_2)$ be given by
\vspace{-0.2cm}
\begin{equation}\label{eqn:+1/2_degree_defect}
\begin{split}
	n &= \left( \cos \frac{\theta[0.5,0.5]}{2}, \sin \frac{\theta[0.5,0.5]}{2} \right), \quad \theta[a,b](x_1,x_2) := \mbox{atan2} \left(\frac{x_2-b}{x_1-a}\right),
\end{split}
\end{equation}
where $\mbox{atan2}$ is the four-quadrant inverse tangent function and brackets $[a,b]$ indicate parameters.  In other words, $n \otimes n$ corresponds to a $+1/2$ degree defect centered at $(0.5,0.5)$.  Next, we set $r[a,b](x_1,x_2) = |(x_1 - a, x_2 - b)|$ and
\vspace{-0.2cm}
\begin{equation}\label{eqn:+1/2_defect_2D_init_cond}
\begin{split}
	\tQ^{0} := \frac{r^2[0.5,0.5]}{r^2[0.5,0.5] + \delta^2} s_{*} \left[ n_i n_j - \delta_{ij}/2 \right]_{i,j=1}^{2},
\end{split}
\end{equation}
where $\delta = \bulketa / 4$; this ensures that $\tQ^{0} \in H^1(\Om;\symmtraceless) \cap C^0(\Om;\symmtraceless)$.  
The final time is $\tfinal = 0.4$ and the time-step is $\dt = 0.004$.

The control parameters in \cref{eq:optconprob} are
\vspace{-0.3cm}
\begin{equation}\label{eqn:+1/2_defect_2D_ctrl_param}
\begin{split}
\beta_{\domcyl} = 1.0, \quad \beta_{\bdycyl} = 0.0, \quad \beta_{\tfinal} = 1.0, \quad \alpha_{\domcyl} = 0.0, \quad \alpha_{\bdycyl} = 0.01,
\end{split}
\end{equation}
and the targets are given by setting $z = \left( \cos \frac{\theta[0.25,0.35]}{2}, \sin \frac{\theta[0.25,0.35]}{2} \right)$ and
\vspace{-0.3cm}
\begin{equation}\label{eqn:+1/2_defect_2D_ctrl_targets}
\begin{split}
	\domtarget = \fintarget = \frac{r^2[0.25,0.35]}{r^2[0.25,0.35] + \delta^2} s_{*} \left[ z_i z_j - \delta_{ij}/2 \right]_{i,j=1}^{2}, \quad \bdytarget = 0.
\end{split}
\end{equation}
In other words, the control objective is to drive $\tQ$ toward a state that has a $+1/2$ degree defect located at the coordinates $(0.25,0.35)$.

In this example, we set $\domcon \equiv 0$, so we only optimize the boundary control $\bdycon$ which we enforce to be time-independent.  The initial guess for optimizing the control is given by setting $u = \left( \cos \theta[0.5,0.5], \sin \theta[0.5,0.5] \right)$ and
\vspace{-0.2cm}
\begin{equation}\label{eqn:+1/2_defect_2D_init_ctrl}
	\bdycon = \frac{r^2[0.5,0.5]}{r^2[0.5,0.5] + \delta^2} s_{*} \left[ u_i u_j - \delta_{ij}/2 \right]_{i,j=1}^{2}.
\end{equation}
Note: we enforce the convex constraint in \cref{eq:con_constr} with a projected gradient method.

\Cref{fig:ctrl_+1/2_defect_2D_optim_history} shows the performance of our gradient descent method.  The $\bdycon$ residual is computed as follows. 
Let $\tP_{\Gm}^{k}$ satisfy $\inner{\tP_{\Gm}^{k}}{\tH_{\Gm}}{L^2(\Gm)} = - \reducedobj_{0,\bdycon^{k}}'(0,\tH_{\Gm})$, for all $\tH_{\Gm} \in \HGsh$ (note: we treat the control as time-independent here), where $k$ is the optimization iteration.  In other words, $\tP_{\Gm}^{k}$ is the $L^2(\Gm)$ projection of the negative gradient.  Next, let $\Pi_{\Gm}$ be the projection onto the boundary control part of the convex set in \cref{eq:con_constr}.  Then the $\bdycon$ residual, at the $k$-th iteration, is defined as $\| \bdycon^{k} - \Pi_{\Gm} \left( \bdycon^{k} + \tP_{\Gm}^{k} \right) \|_{L^2(\Gm)}$.  
The computed boundary controls $\bdycon$ at later iterations do not exhibit any active set, i.e. the inequality constraint is not active.  However, we found that removing the constraint yielded an optimal $\bdycon$ that was \emph{not physical}, i.e. the eigenvalues of $\bdycon$ were outside the physical range (recall the discussion around \cref{eq:con_constr}).  Thus, it is necessary to enforce the inequality constraint during the line-search.
\begin{figure}%
\begin{subfigure}{.49\textwidth}
\centering
\includegraphics[width=.85\linewidth]{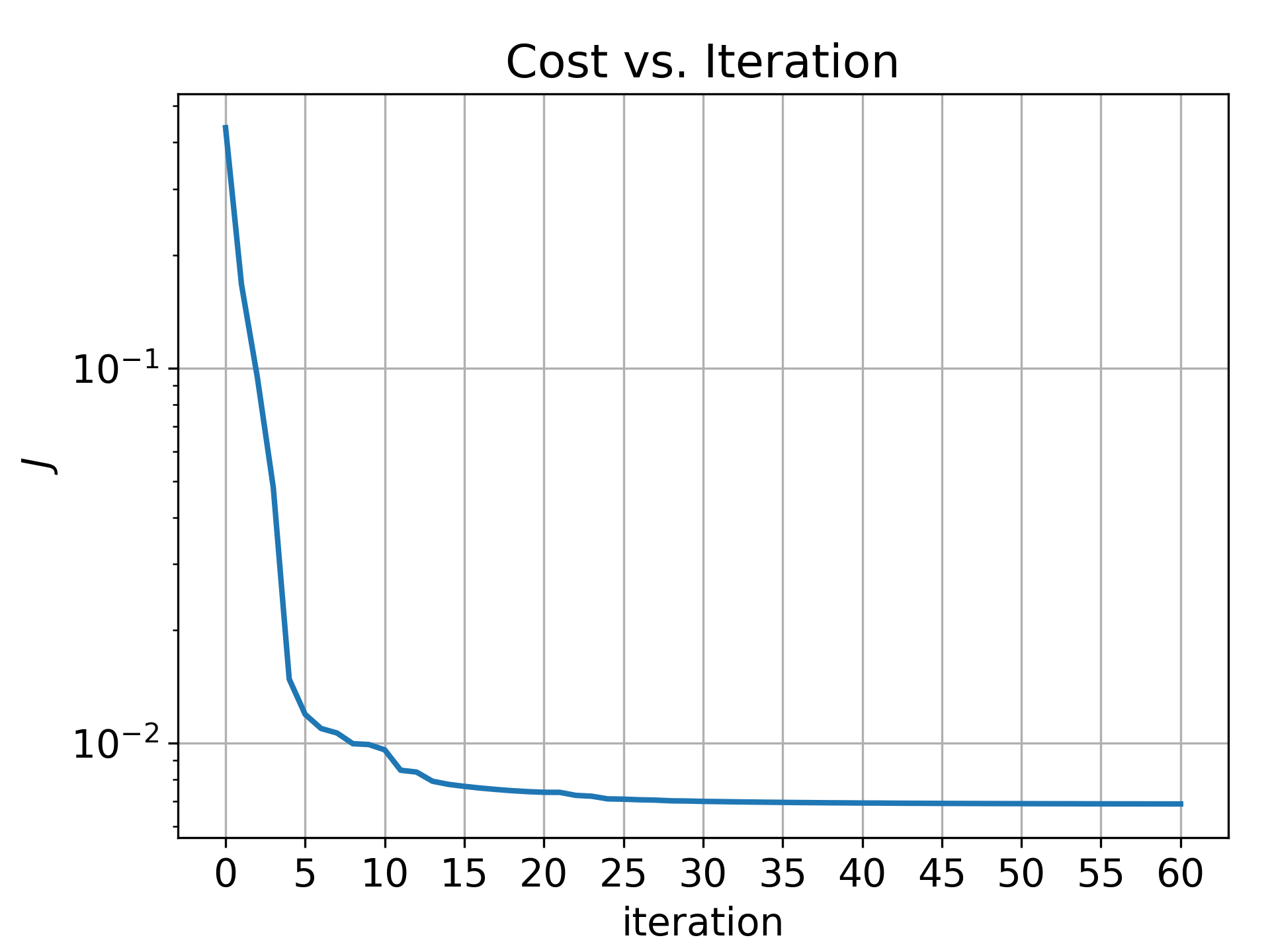}
\end{subfigure}
\begin{subfigure}{.49\textwidth}
\centering
\includegraphics[width=.85\linewidth]{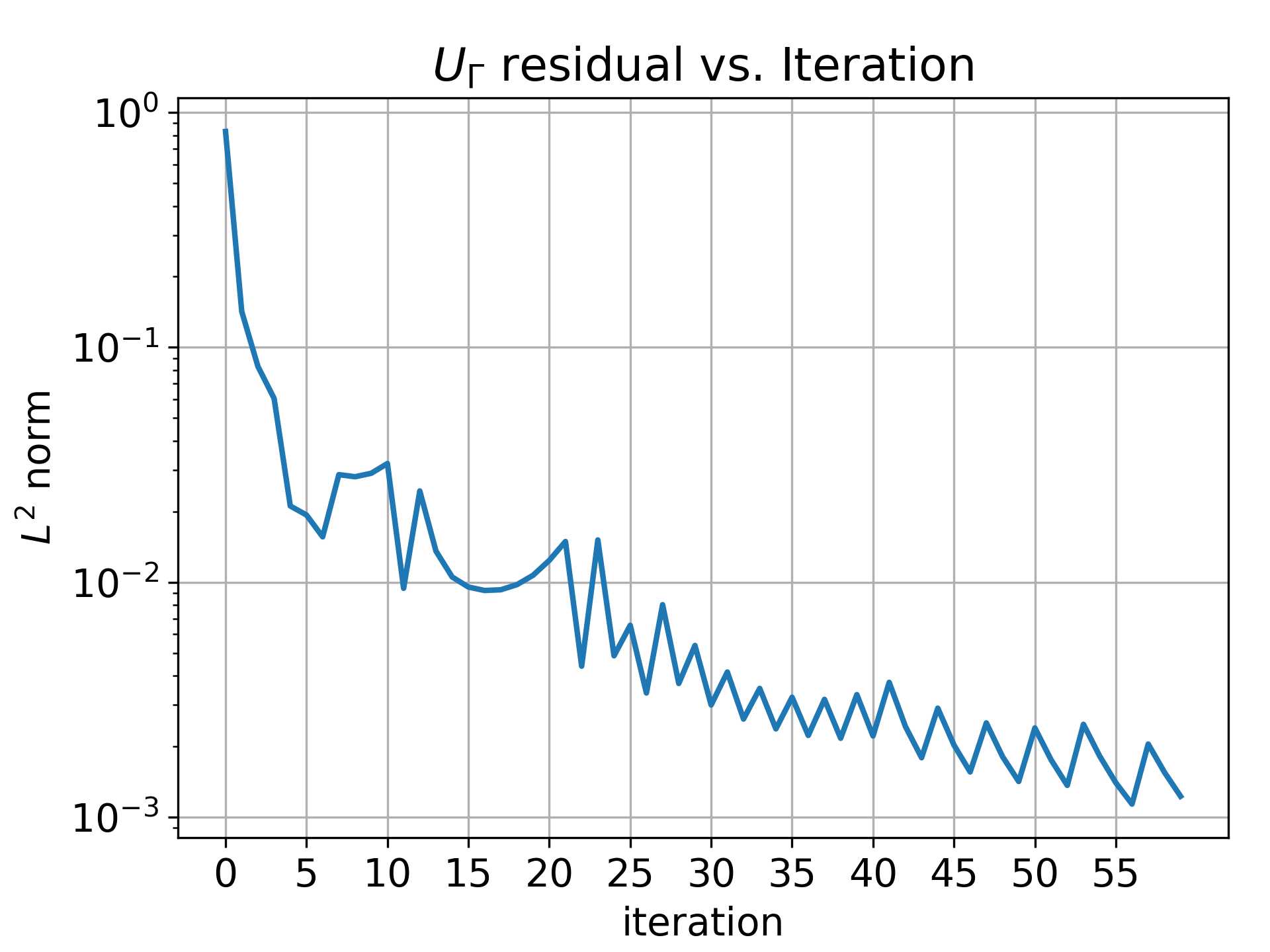}
\end{subfigure}

\caption{Optimization history (\cref{sec:ctrl_+1/2_defect_2D}).  The $\bdycon$ residual is described in the text.
}
\label{fig:ctrl_+1/2_defect_2D_optim_history}
\vspace{-0.2cm}
\end{figure}

\Cref{fig:ctrl_+1/2_defect_2D_target_bdycon} shows the target $\domtarget$ and optimized boundary control $\bdycon$.  
\begin{figure}%
\begin{subfigure}{.49\textwidth}
\centering
\includegraphics[width=.98\linewidth]{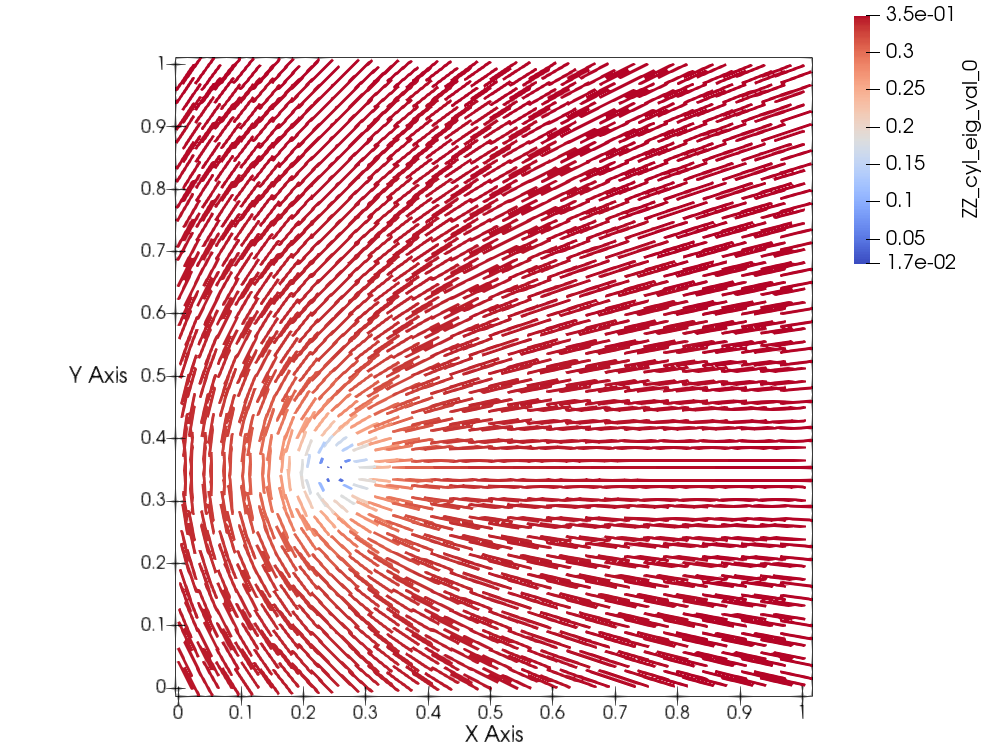}
\end{subfigure}
\begin{subfigure}{.49\textwidth}
\centering
\includegraphics[width=.98\linewidth]{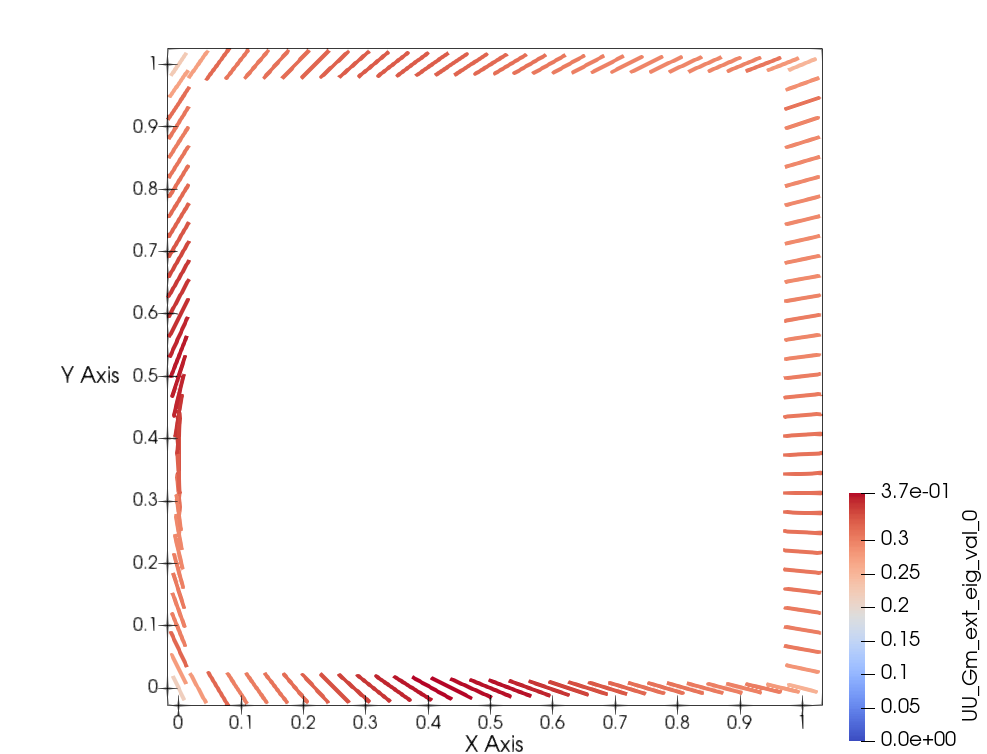}
\end{subfigure}

\caption{Target $\domtarget$ (left) and optimized boundary control $\bdycon$ (right) (\cref{sec:ctrl_+1/2_defect_2D}).  We visualize $\domtarget$ by plotting line segments that correspond to the eigenvector of $\domtarget$ with maximum eigenvalue; $\bdycon$ is visualized similarly.  Note how the boundary control mimics the boundary conditions of the target.
}
\label{fig:ctrl_+1/2_defect_2D_target_bdycon}
\vspace{-0.2cm}
\end{figure}
\Cref{fig:ctrl_+1/2_defect_2D_Q_0_tf} shows the initial and final state of $\tQ$ that clearly demonstrates the efficacy of the control.
\begin{figure}%
\begin{subfigure}{.49\textwidth}
\centering
\includegraphics[width=.98\linewidth]{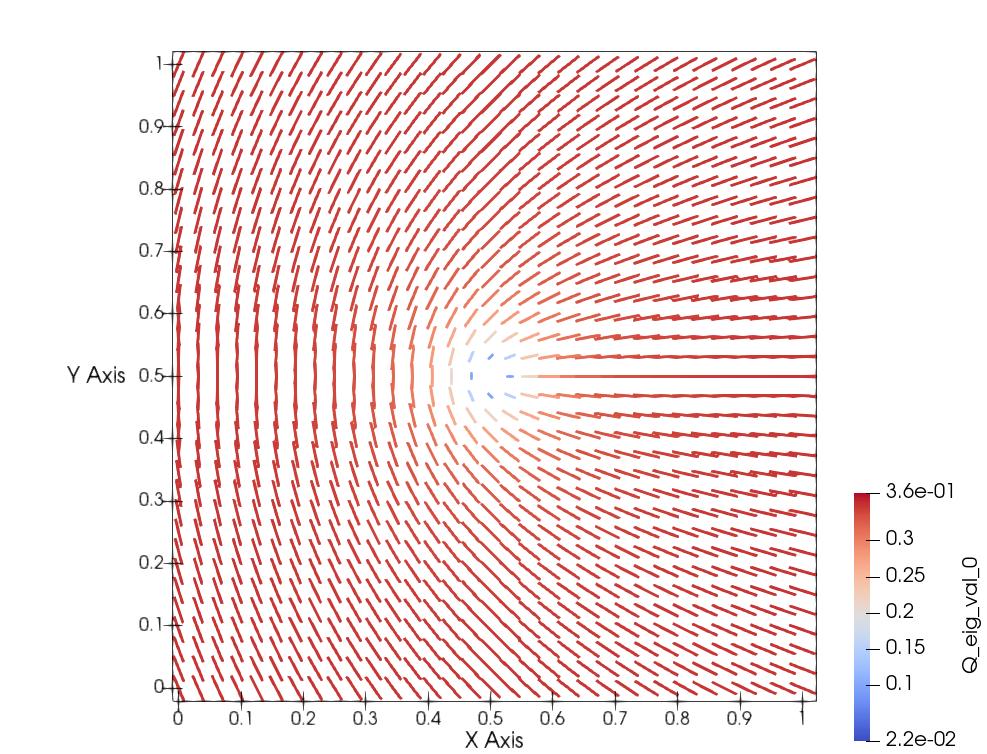}
\end{subfigure}
\begin{subfigure}{.49\textwidth}
\centering
\includegraphics[width=.98\linewidth]{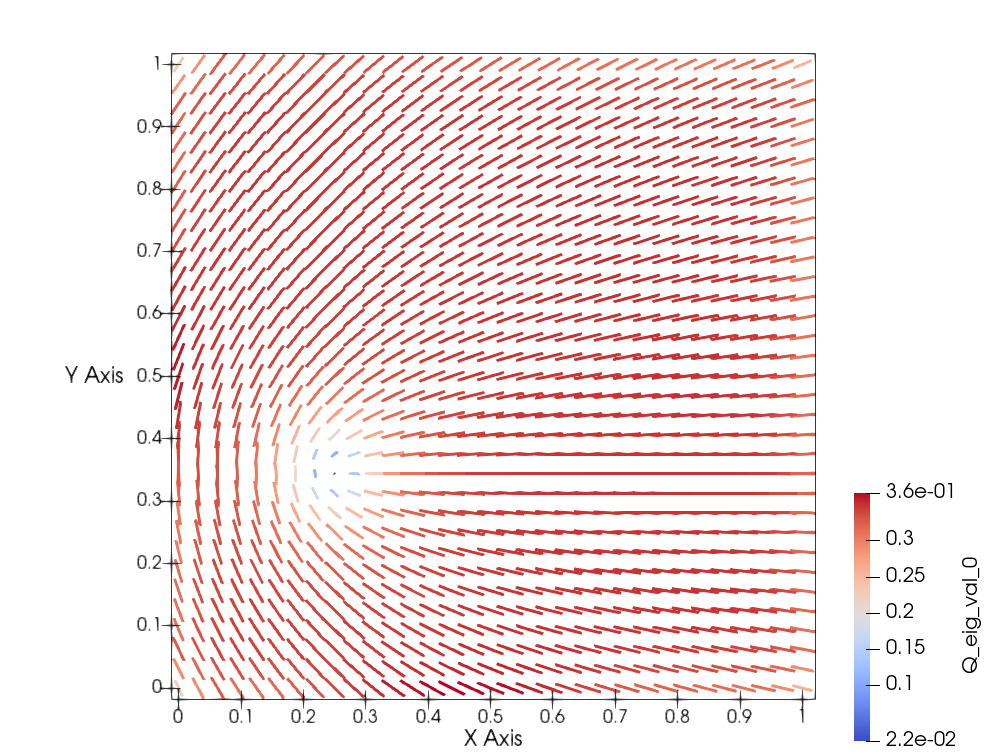}
\end{subfigure}

\caption{Initial state for $\tQ(t=0)$ (left) and final state for $\tQ(t=\tfinal)$ (right) (\cref{sec:ctrl_+1/2_defect_2D}). 
Line segments correspond to the eigenvector of $\tQ$ with maximum eigenvalue; the color scale is based on the maximum eigenvalue.  
The position of the point defect (at the final time) is almost exactly the same as in the target.
}
\label{fig:ctrl_+1/2_defect_2D_Q_0_tf}
\vspace{-0.2cm}
\end{figure}

\subsection{Prevent $+1/2$ and $-1/2$ degree point defects from annihilating in two dimensions}\label{sec:prevent_+/-_1/2_defects_annihlate_2D}

Most of the parameters are the same as in \cref{sec:ctrl_+1/2_defect_2D} with the following modifications.  
The initial condition is given by first defining:
\vspace{-0.2cm}
\begin{equation}\label{eqn:prevent_+/-_1/2_degree_defect}
\begin{split}
	n &:= \left( \cos \left( \frac{\theta[0.4, 0.505]}{2} + \frac{\pi}{2} \right), \sin \left( \frac{\theta[0.4, 0.505]}{2} + \frac{\pi}{2} \right) \right), \\
	m &:= \left( \cos -\frac{\theta[0.6, 0.495]}{2}, \sin -\frac{\theta[0.6, 0.495]}{2} \right),
\end{split}
\end{equation}
i.e. $n \otimes n$ corresponds to a $+1/2$ degree defect centered at $(0.4, 0.505)$ and $m \otimes m$ corresponds to a $-1/2$ degree defect centered at $(0.6, 0.495)$.  Then, we set
\vspace{-0.2cm}
\begin{equation}\label{eqn:+/-_1/2_defect_Q_funcs}
\begin{split}
	\tQ_{n}^{0} &:= \frac{r^2[0.4, 0.505]}{r^2[0.4, 0.505] + \delta^2} s_{*} \left[ n_i n_j - \delta_{ij}/2 \right]_{i,j=1}^{2}, \\
	\tQ_{m}^{0} &:= \frac{r^2[0.6, 0.495]}{r^2[0.6, 0.495] + \delta^2} s_{*} \left[ m_i m_j - \delta_{ij}/2 \right]_{i,j=1}^{2},
\end{split}
\end{equation}
where $\delta = \bulketa / 4$; this ensures that $\tQ_{n}^{0}, \tQ_{m}^{0} \in H^1(\Om;\symmtraceless) \cap C^0(\Om;\symmtraceless)$.  Then, the initial condition is given by the following interpolation:
\vspace{-0.2cm}
\begin{equation}\label{eqn:prevent_+/-_1/2_defect_2D_init_cond}
\begin{split}
	\tQ^{0} &:= (1 - x) \tQ_{n}^{0} + x \tQ_{m}^{0}.
\end{split}
\end{equation}
%The final time is $\tfinal = 0.4$ and the time-step is $\dt = 0.004$.

The control parameters in \cref{eq:optconprob} are the same as in 
\cref{eqn:+1/2_defect_2D_ctrl_param}, and the targets $\domtarget$, $\fintarget$ have the same form as \cref{eqn:prevent_+/-_1/2_defect_2D_init_cond}, except the $+1/2$ defect is placed at $(0.2, 0.6)$ and the $-1/2$ defect is placed at $(0.8, 0.4)$.  Note that $\bdytarget = 0$ plays no role.  In other words, the control objective is to drive $\tQ$ toward a stable configuration of a $+1/2$ and $-1/2$ defect.  In this example, we set $\domcon \equiv 0$, so we only optimize the boundary control $\bdycon$ which we enforce to be time-independent.  The initial guess for optimizing the control is the constant tensor $\bdycon = s_{*} \left[ \hat{n}_i \hat{n}_j - \delta_{ij}/2 \right]_{i,j=1}^{2}$, where $\hat{n} = (1,0)$.  In this case, the $\tQ$ state evolves toward a constant state identical to the initial boundary control, i.e. the two initial defects annihilate.

In this example, we modify the inequality constraint in \cref{eq:con_constr} to be $|\tP| \leq 0.6$ on $\Gm$.  \Cref{fig:ctrl_+/-_1/2_defect_2D_optim_history} shows the performance of our gradient descent method.  The $\bdycon$ residual is computed as in \cref{sec:ctrl_+1/2_defect_2D}.  
The computed boundary control $\bdycon$ does exhibit an active set.  Indeed, it was necessary to lower the bound to $0.6$ in order to ensure that the computed control satisfied the eigenvalue bounds described in \cref{ssec:ldg}, which in two dimensions is $-1/2 \leq \lambda_{i} (\bdycon) \leq 1/2$, for $i=1,2$.  This further emphasizes that the inequality constraint is needed to prevent computing minimizers of the objective functional that are not physical (see the discussion in \cref{sec:ctrl_+1/2_defect_2D}).

\begin{figure}%
\begin{subfigure}{.49\textwidth}
\centering
\includegraphics[width=.85\linewidth]{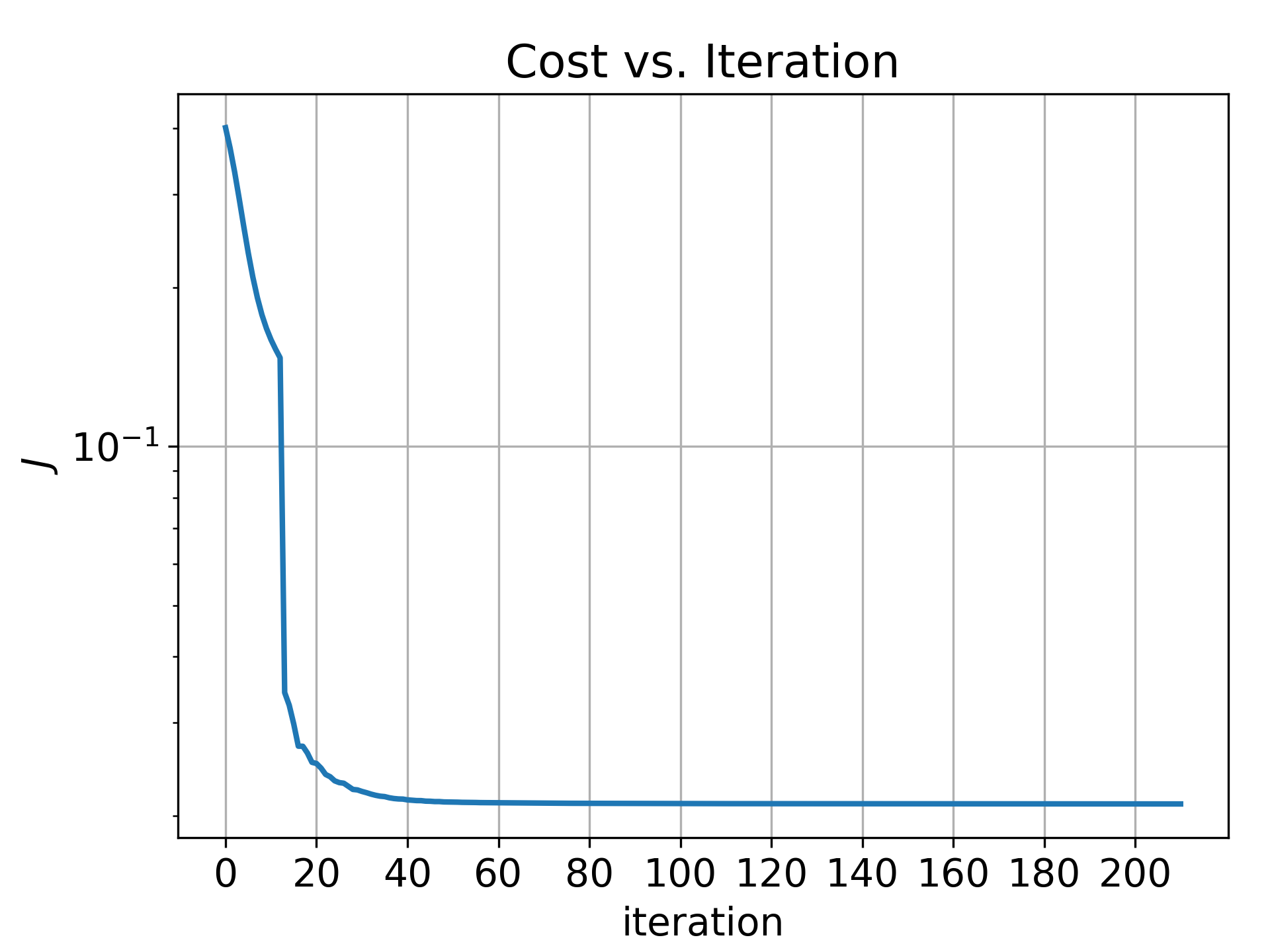}
\end{subfigure}
\begin{subfigure}{.49\textwidth}
\centering
\includegraphics[width=.85\linewidth]{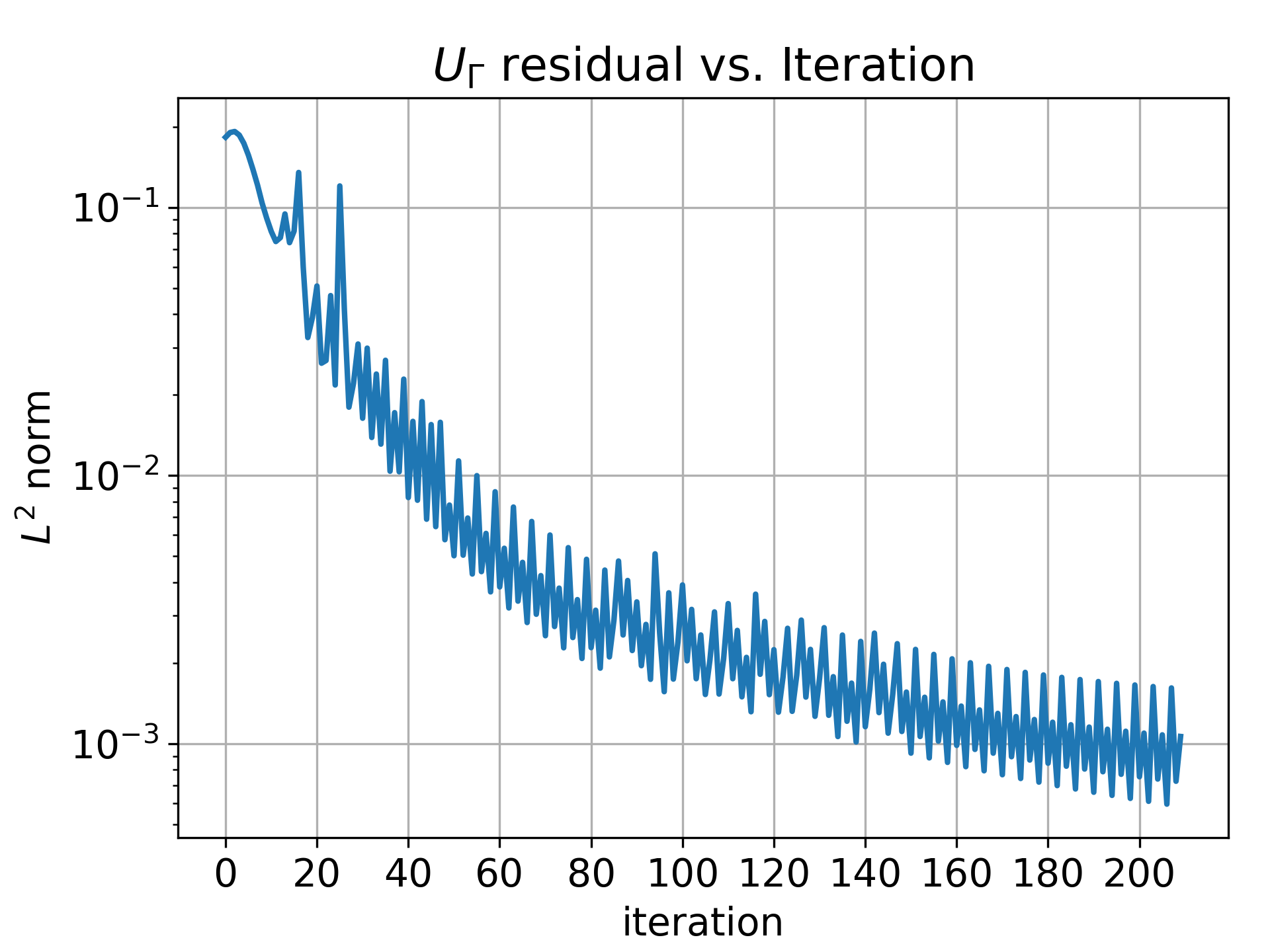}
\end{subfigure}

\caption{Optimization history (\cref{sec:prevent_+/-_1/2_defects_annihlate_2D}).  The $\bdycon$ residual is described in the text.
}
\label{fig:ctrl_+/-_1/2_defect_2D_optim_history}
\vspace{-0.2cm}
\end{figure}

\Cref{fig:ctrl_+/-_1/2_defect_2D_target_bdycon} shows the target $\domtarget$ and optimized boundary control $\bdycon$; note that the maximum value of $\lambda_{1}(\bdycon)$ is $0.42$.  
\begin{figure}%
\begin{subfigure}{.49\textwidth}
\centering
\includegraphics[width=.98\linewidth]{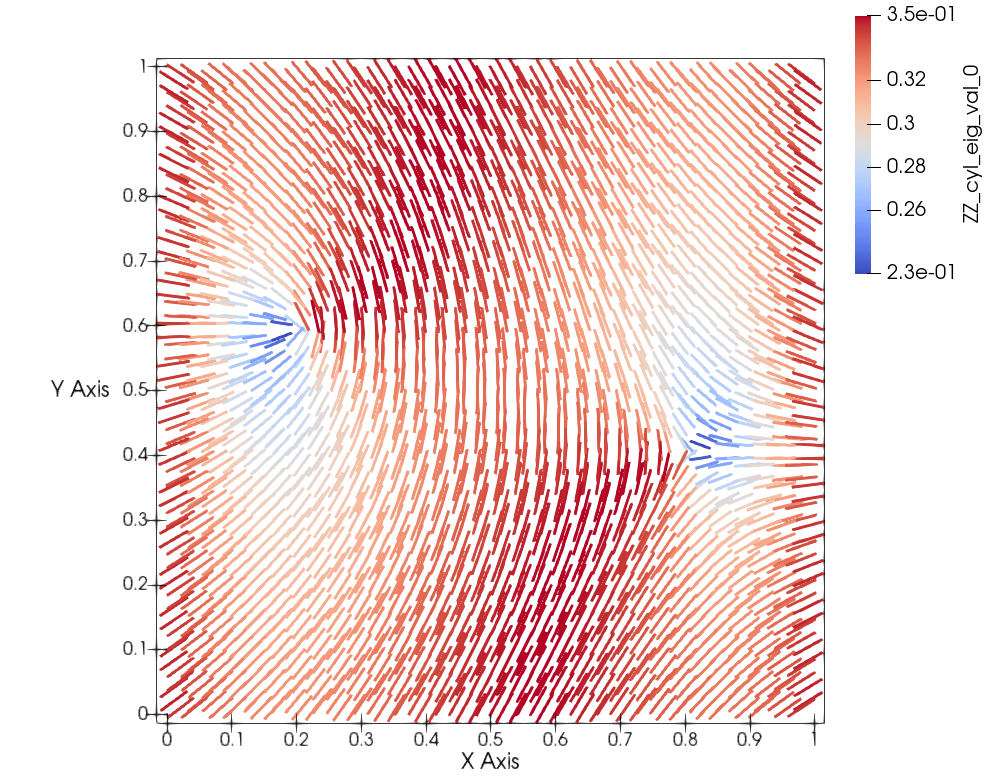}
\end{subfigure}
\begin{subfigure}{.49\textwidth}
\centering
\includegraphics[width=.98\linewidth]{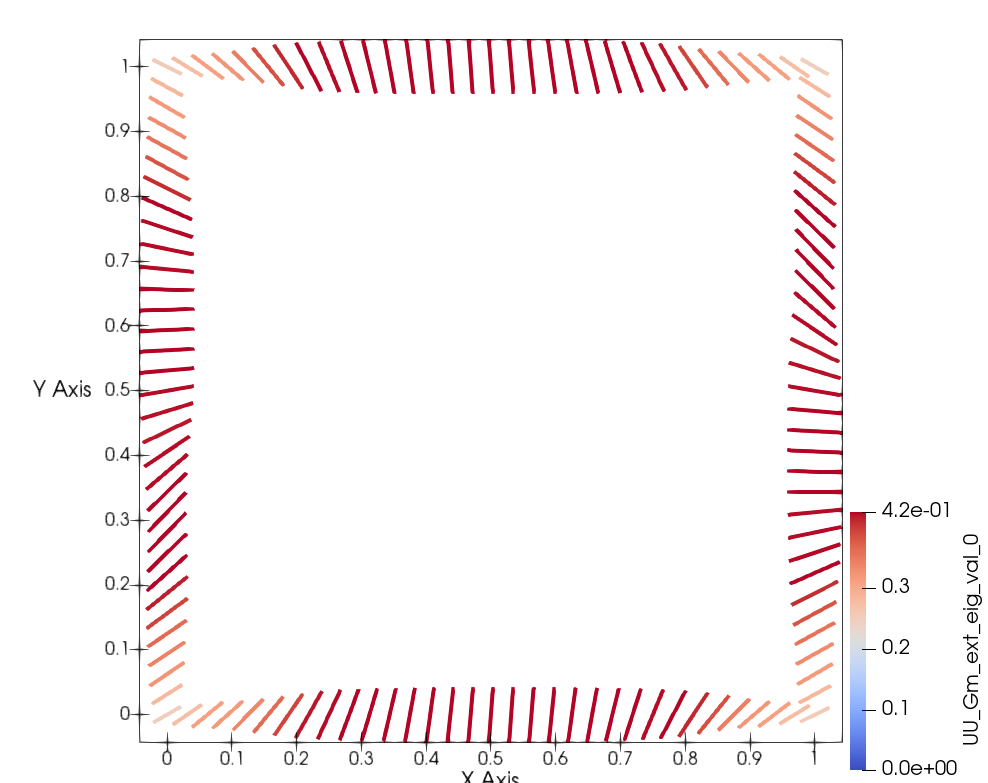}
\end{subfigure}

\caption{Target $\domtarget$ (left) and optimized boundary control $\bdycon$ (right) (\cref{sec:prevent_+/-_1/2_defects_annihlate_2D}).  We visualize $\domtarget$ by plotting line segments that correspond to the eigenvector of $\domtarget$ with maximum eigenvalue; $\bdycon$ is visualized similarly.  Note how the boundary control mimics the boundary conditions of the target.
}
\label{fig:ctrl_+/-_1/2_defect_2D_target_bdycon}
\vspace{-0.2cm}
\end{figure}
\Cref{fig:ctrl_+/-_1/2_defect_2D_Q_0_tf} shows the initial and final state of $\tQ$ that clearly demonstrates the efficacy of the control.
\begin{figure}%
\begin{subfigure}{.49\textwidth}
\centering
\includegraphics[width=.98\linewidth]{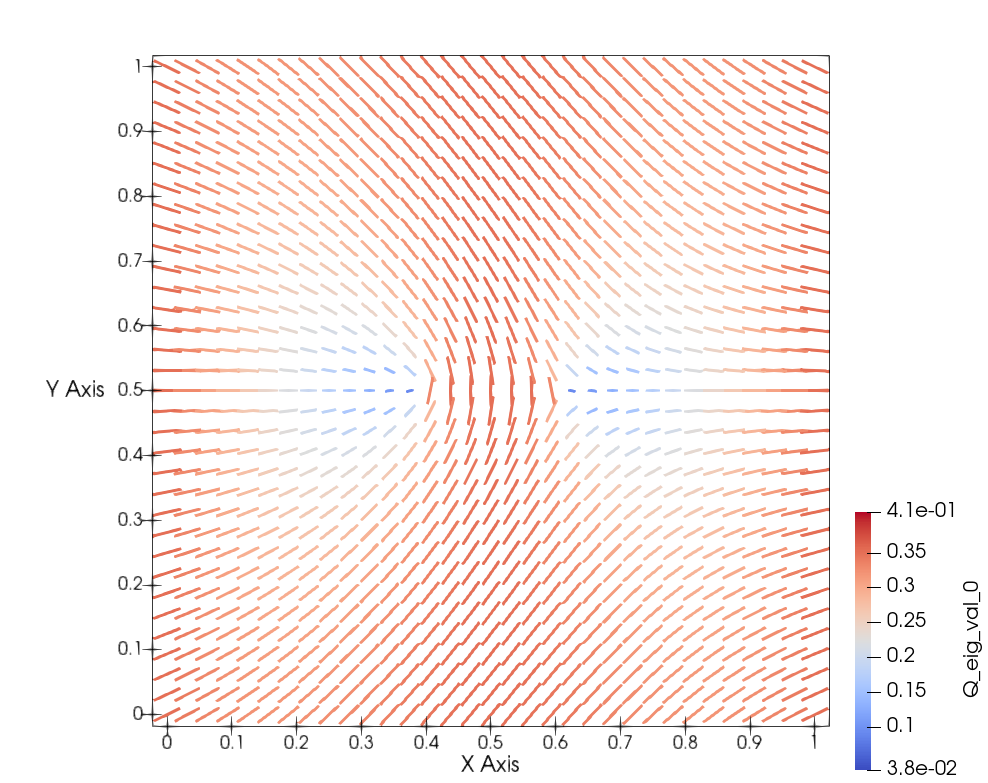}
\end{subfigure}
\begin{subfigure}{.49\textwidth}
\centering
\includegraphics[width=.98\linewidth]{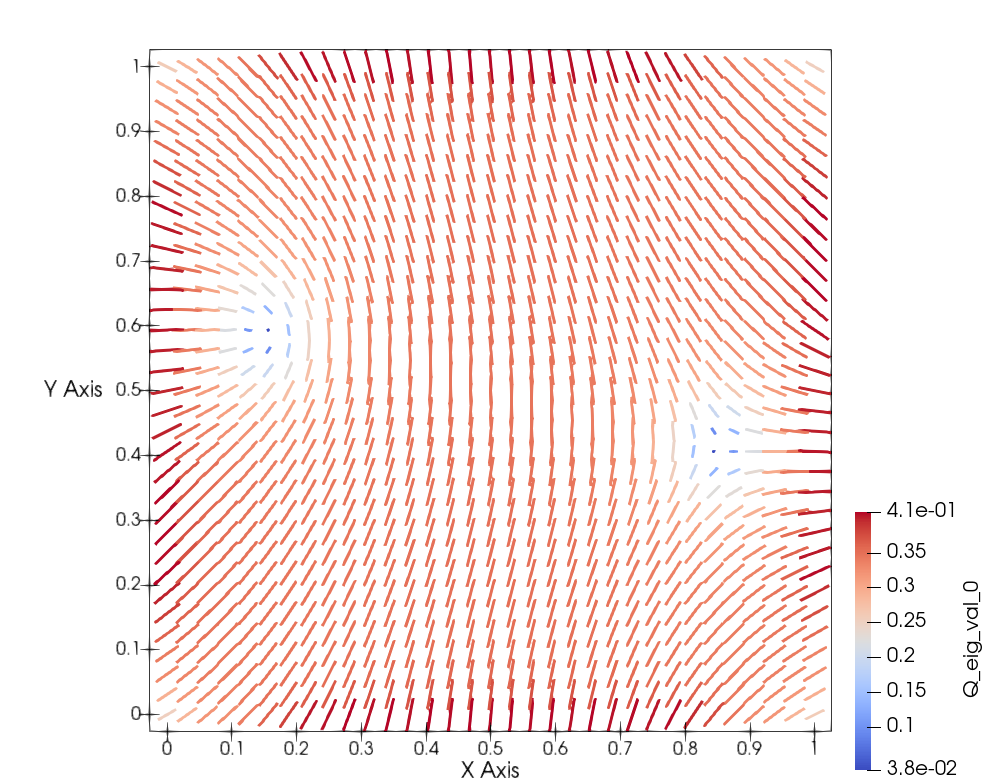}
\end{subfigure}

\caption{Initial state for $\tQ(t=0)$ (left) and final state for $\tQ(t=\tfinal)$ (right) (\cref{sec:prevent_+/-_1/2_defects_annihlate_2D}). 
Line segments correspond to the eigenvector of $\tQ$ with maximum eigenvalue; the color scale is based on the maximum eigenvalue.  
The position of the point defect (at the final time) is almost exactly the same as in the target.
}
\label{fig:ctrl_+/-_1/2_defect_2D_Q_0_tf}
\vspace{-0.2cm}
\end{figure}

\subsection{Control of a $+1/2$ line defect in three dimensions}\label{sec:ctrl_+1/2_curve_defect_3D}

The domain is the unit cube $\Om = (0,1)^3$ and the parameters of the forward problem are as follows.  The coefficients of the double well in \cref{eqn:Landau-deGennes_bulk_potential} are
\vspace{-0.2cm}
\begin{equation}\label{eqn:+1/2_curve_defect_3D_bulk_pot_coefs}
\begin{split}
	\bulkK = 1, \quad \bulkA = 7.5021037403, \quad \bulkB = 60.975813166, \quad \bulkC = 66.519068908,
\end{split}
\end{equation}
and $\bulkfunc(\tQ)$ has a global minimum at $\tQ_{*} = s_{*} \left[ n_i n_j - \delta_{ij}/3 \right]_{i,j=1}^{3}$, where $n \in \R^3$ is any unit vector, and $s_{*} = 0.700005531$.  The other coefficients are given by $\bulketa = 0.2$, $\domcoef = 0$, $\anchcoef = 100$.

The initial condition was defined as follows.  First, let $n = n(x_1,x_2)$ be given by
\vspace{-0.2cm}
\begin{equation}\label{eqn:+1/2_degree_curve_defect}
\begin{split}
	n &:= \left( \cos \frac{\theta[0.5,0.5]}{2}, \sin \frac{\theta[0.5,0.5]}{2}, 0 \right), \quad \theta[a,b](x_1,x_2) := \mbox{atan2} \left(\frac{x_2-b}{x_1-a}\right),
\end{split}
\end{equation}
similar to \cref{eqn:+1/2_degree_defect}.  In other words, $n \otimes n$ corresponds to a $+1/2$ degree defect, in any plane parallel to the $x_3=0$ plane, centered at $(0.5,0.5)$.  Then, we have
\vspace{-0.2cm}
\begin{equation}\label{eqn:+1/2_curve_defect_3D_init_cond}
\begin{split}
	\tQ^{0} := \frac{r^2[0.5,0.5]}{r^2[0.5,0.5] + \delta^2} s_{*} \left[ n_i n_j - \delta_{ij}/3 \right]_{i,j=1}^{3},
\end{split}
\end{equation}
where $\delta = \bulketa / 4$; this ensures that $\tQ^{0} \in H^1(\Om;\symmtraceless) \cap C^0(\Om;\symmtraceless)$.  
The final time is $\tfinal = 0.3$ and the time-step is $\dt = 0.006$.

The control parameters in \cref{eq:optconprob} are the same as in \cref{eqn:+1/2_defect_2D_ctrl_param}.  
%\begin{equation}\label{eqn:+1/2_curve_defect_3D_ctrl_param}
%\begin{split}
%\beta_{\domcyl} = 1.0, \quad \beta_{\bdycyl} = 0.0, \quad \beta_{\tfinal} = 1.0, \quad \alpha_{\domcyl} = 0.0, \quad \alpha_{\bdycyl} = 0.01,
%\end{split}
%\end{equation}
The targets are defined through a parameterized curve in $\R^3$, denoted $\left( \tilde{x}_1 (\xi), \tilde{x}_2 (\xi), \tilde{x}_3 (\xi) \right)$, given by
\vspace{-0.2cm}
\begin{equation}\label{eqn:+1/2_curve_defect_3D_param_curve}
\begin{split}
	f(\xi) &:= l_2 \xi^2 + l_3 \xi^3, \quad l_2 = 3 c_0, ~ l_3 = -2 c_0, ~ c_0 = 0.6, \\
	\left( \tilde{x}_1 (\xi), \tilde{x}_2 (\xi), \tilde{x}_3 (\xi) \right) &:= \left( f(\xi) + 0.2 , f(\xi) + 0.2, \xi \right), \text{ for } 0 \leq \xi \leq 1.
\end{split}
\end{equation}
Next, we define $\tilde{r}^2(x_1,x_2,x_3) := \left( x_1-\tilde{x}_1 (x_3) \right)^2 + \left( x_2-\tilde{x}_2 (x_3) \right)^2$,
\vspace{-0.2cm}
\begin{equation}\label{eqn:+1/2_curve_3D_angle_and_vec}
\begin{split}
	\tilde{\theta}(x_1,x_2,x_3) &:= \mbox{atan2} \left(\frac{x_2-\tilde{x}_2 (x_3)}{x_1-\tilde{x}_1 (x_3)}\right), \quad z = \left( \cos \frac{\tilde{\theta}}{2}, \sin \frac{\tilde{\theta}}{2} , 0 \right),
\end{split}
\end{equation}
and the targets are given by
\vspace{-0.2cm}
\begin{equation}\label{eqn:+1/2_curve_defect_3D_ctrl_targets}
\begin{split}
	\domtarget = \fintarget = \frac{\tilde{r}^2}{\tilde{r}^2 + \delta^2} s_{*} \left[ z_i z_j - \delta_{ij}/3 \right]_{i,j=1}^{3}, \quad \bdytarget = 0.
\end{split}
\end{equation}
In other words, the control objective is to drive $\tQ$ toward a state that has a $+1/2$ degree defect, with respect to the $x_3 = l_0$ plane, located at $(\tilde{x}_1 (l_0), \tilde{x}_2 (l_0), l_0)$.

In this example, we set $\domcon \equiv 0$, so we only optimize the boundary control $\bdycon$ which we enforce to be time-independent.  The initial guess for optimizing the control is given by setting $\bdycon = \tQ^{0}$.

\Cref{fig:ctrl_+1/2_curve_defect_3D_optim_history} shows the performance of our gradient descent method.  The $\bdycon$ residual is computed as in \cref{sec:ctrl_+1/2_defect_2D}.  
The computed boundary controls $\bdycon$ at later iterations do not exhibit any active set, i.e. the inequality constraint is not active.

\begin{figure}%
\begin{subfigure}{.49\textwidth}
\centering
\includegraphics[width=.85\linewidth]{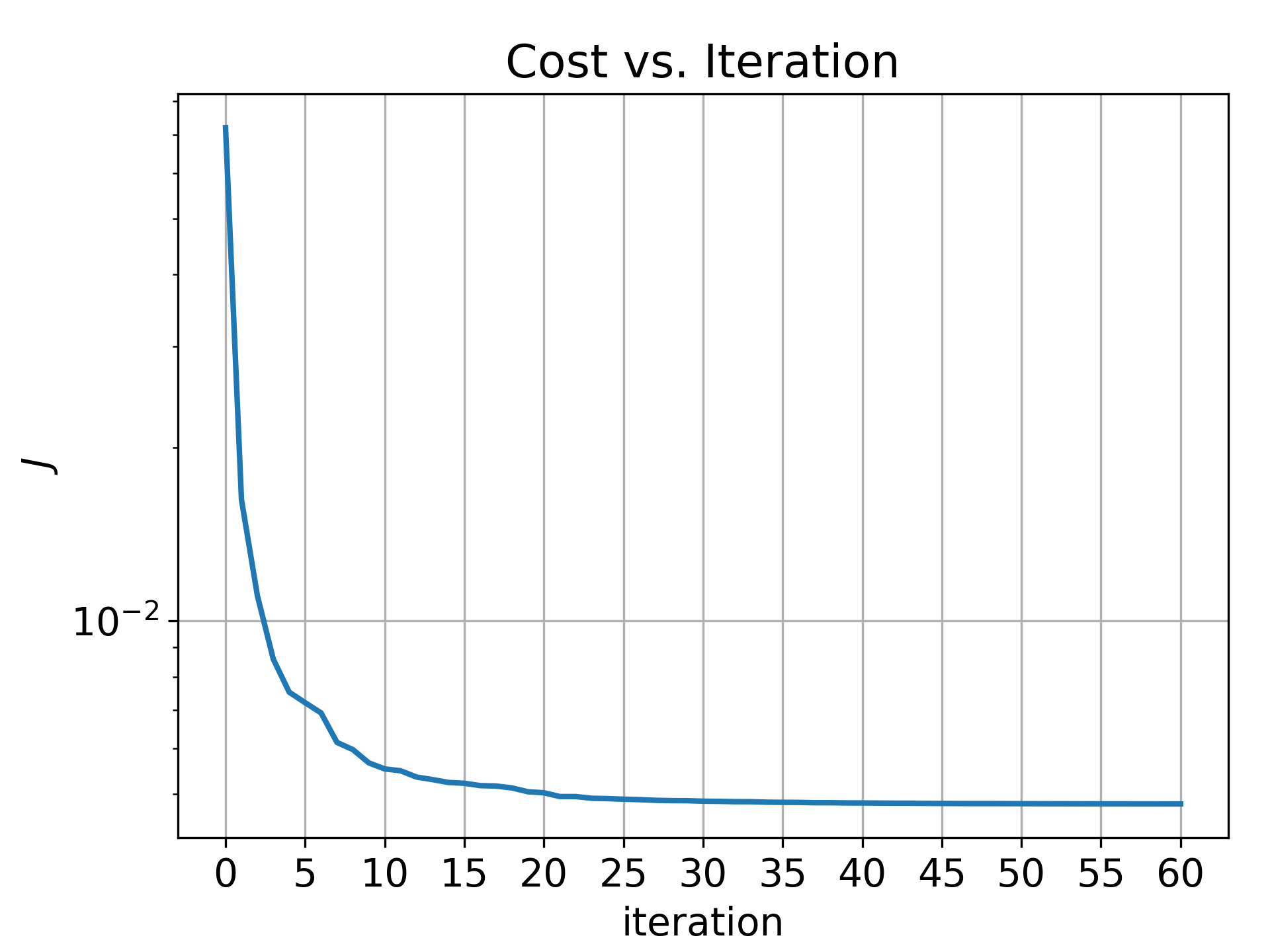}
\end{subfigure}
\begin{subfigure}{.49\textwidth}
\centering
\includegraphics[width=.85\linewidth]{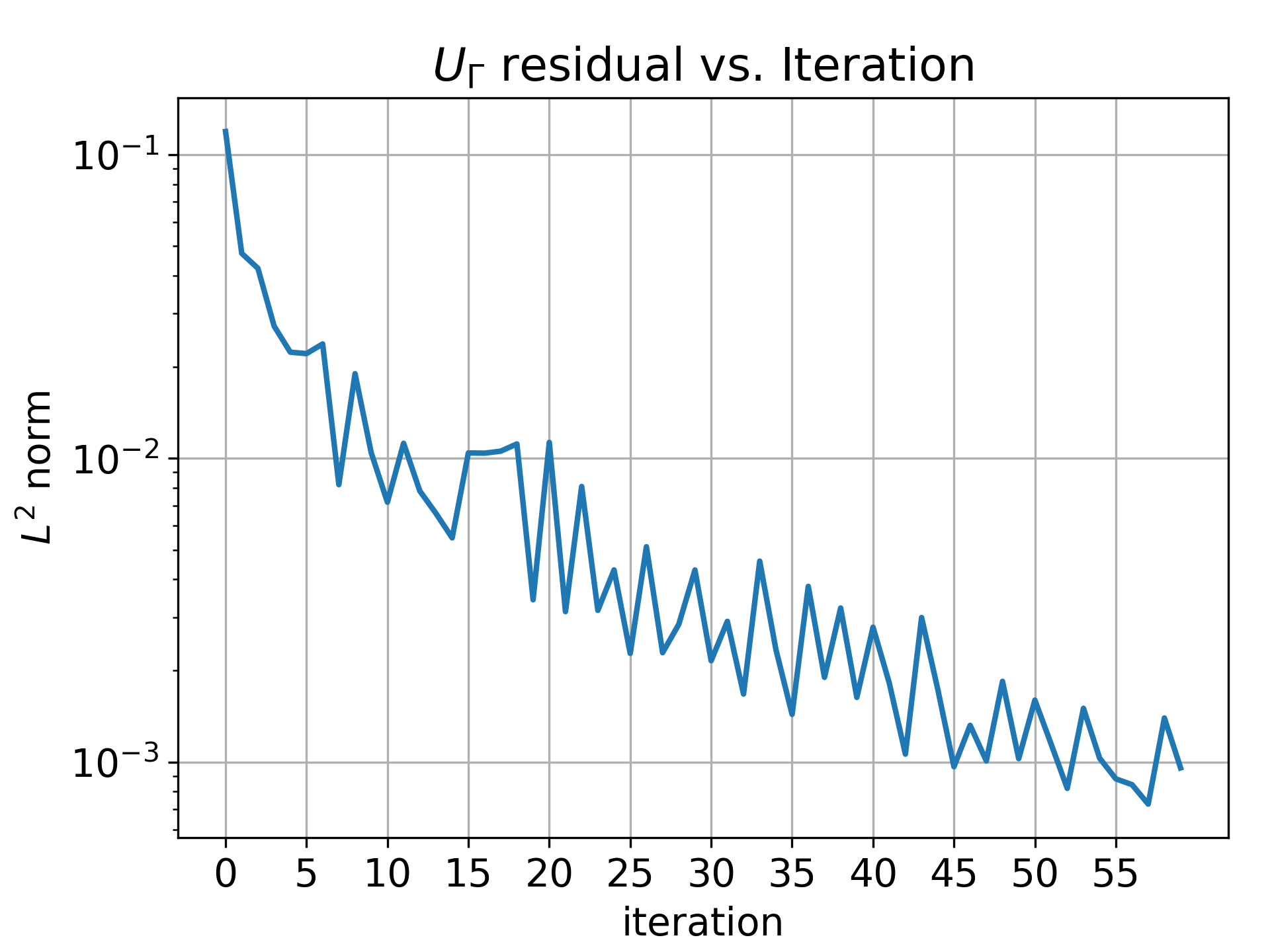}
\end{subfigure}

\caption{Optimization history (\cref{sec:ctrl_+1/2_curve_defect_3D}).  The $\bdycon$ residual is described in the text.
}
\label{fig:ctrl_+1/2_curve_defect_3D_optim_history}
\vspace{-0.2cm}
\end{figure}

\Cref{fig:ctrl_+1/2_curve_defect_3D_target_bdycon} shows the target $\domtarget$ and optimized boundary control $\bdycon$.  We note, however, that the most negative eigenvalue, $\lambda_{3} (\tQ)$ (not plotted), is approximately $-0.33$ at the core of the defect in $\bdycon$ on the $x_{3} = 1$ side of the cube (see middle plot of \cref{fig:ctrl_+1/2_curve_defect_3D_target_bdycon}).  
Again, it is necessary to enforce the inequality constraint during the line-search in order to prevent computing minimizers of the objective functional that are not physical (see the discussion in \cref{sec:ctrl_+1/2_defect_2D}).  
\begin{figure}%
\begin{subfigure}{.32\textwidth}
\centering
\includegraphics[width=.98\linewidth]{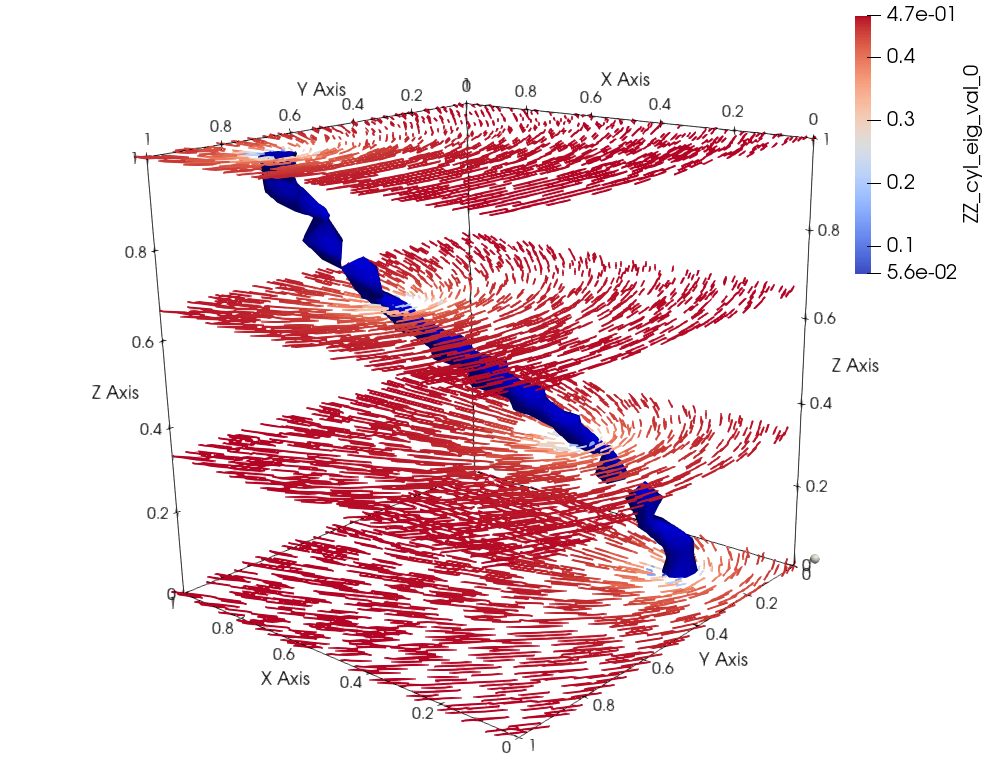}
\end{subfigure}
\begin{subfigure}{.32\textwidth}
\centering
\includegraphics[width=.98\linewidth]{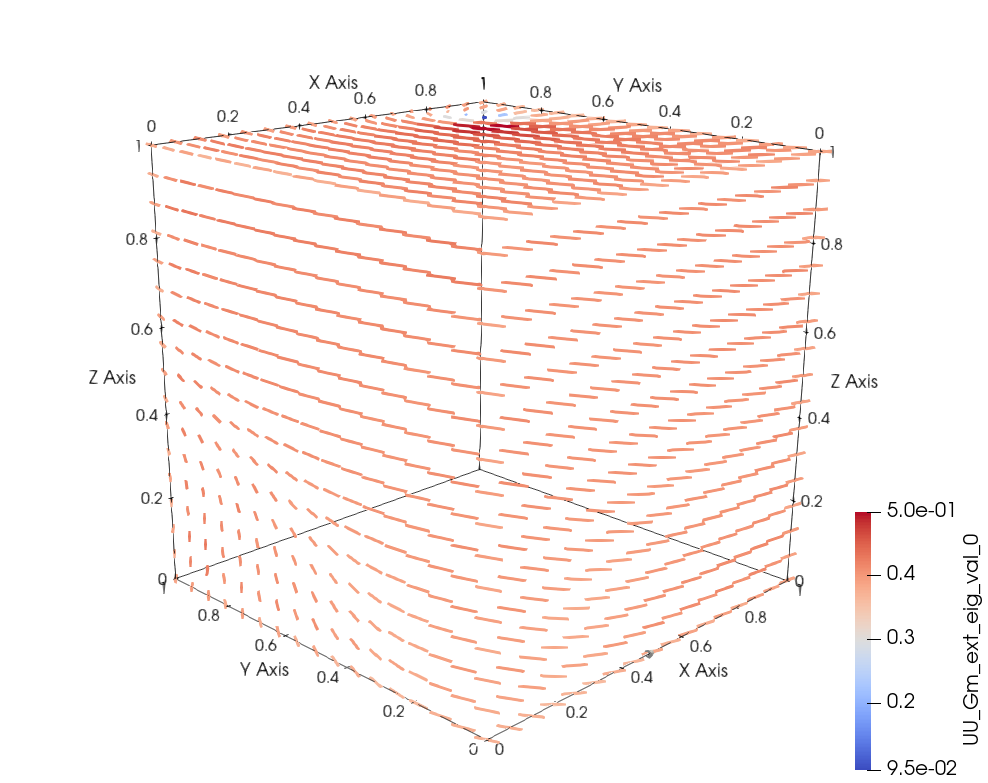}
\end{subfigure}
\begin{subfigure}{.32\textwidth}
\centering
\includegraphics[width=.98\linewidth]{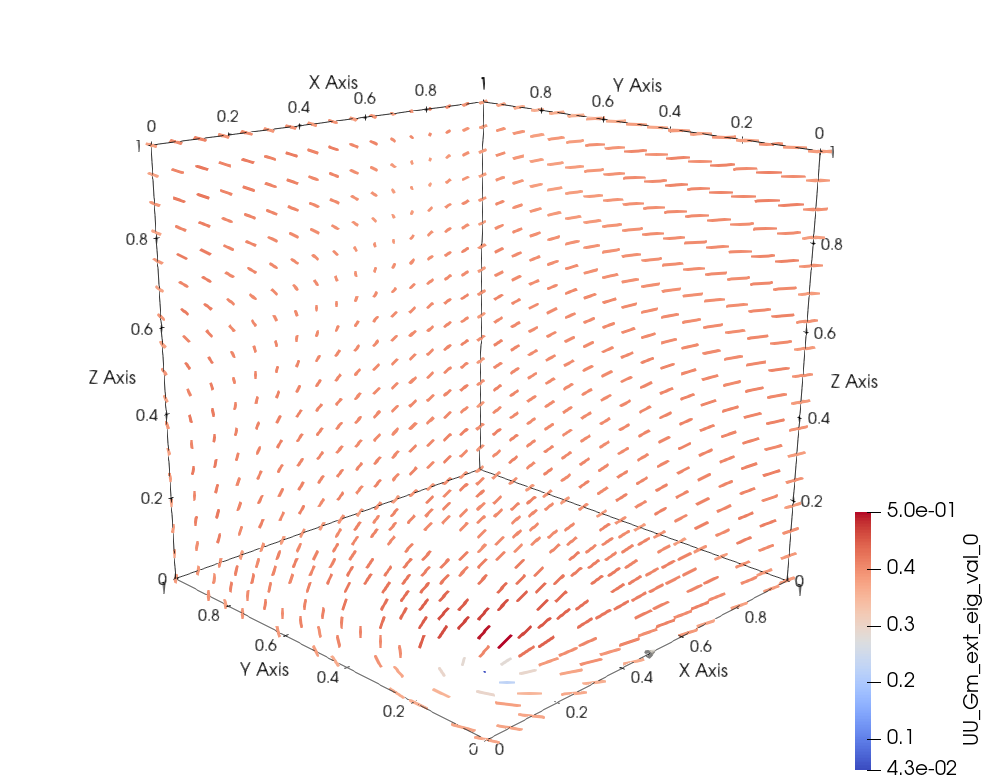}
\end{subfigure}

\caption{Target $\domtarget$ (left) and optimized boundary control $\bdycon$ (middle, right) (\cref{sec:ctrl_+1/2_curve_defect_3D}).  We visualize $\domtarget$ by plotting line segments that correspond to the eigenvector of $\domtarget$ with maximum eigenvalue; $\bdycon$ is visualized similarly.  Middle (Right) view shows the front (back) three faces.
}
\label{fig:ctrl_+1/2_curve_defect_3D_target_bdycon}
\vspace{-0.3cm}
\end{figure}
\Cref{fig:ctrl_+1/2_curve_defect_3D_Q_0_tf} shows the initial and final state of $\tQ$ that clearly demonstrates the efficacy of the control.
\begin{figure}%
\begin{subfigure}{.32\textwidth}
\centering
\includegraphics[width=.99\linewidth]{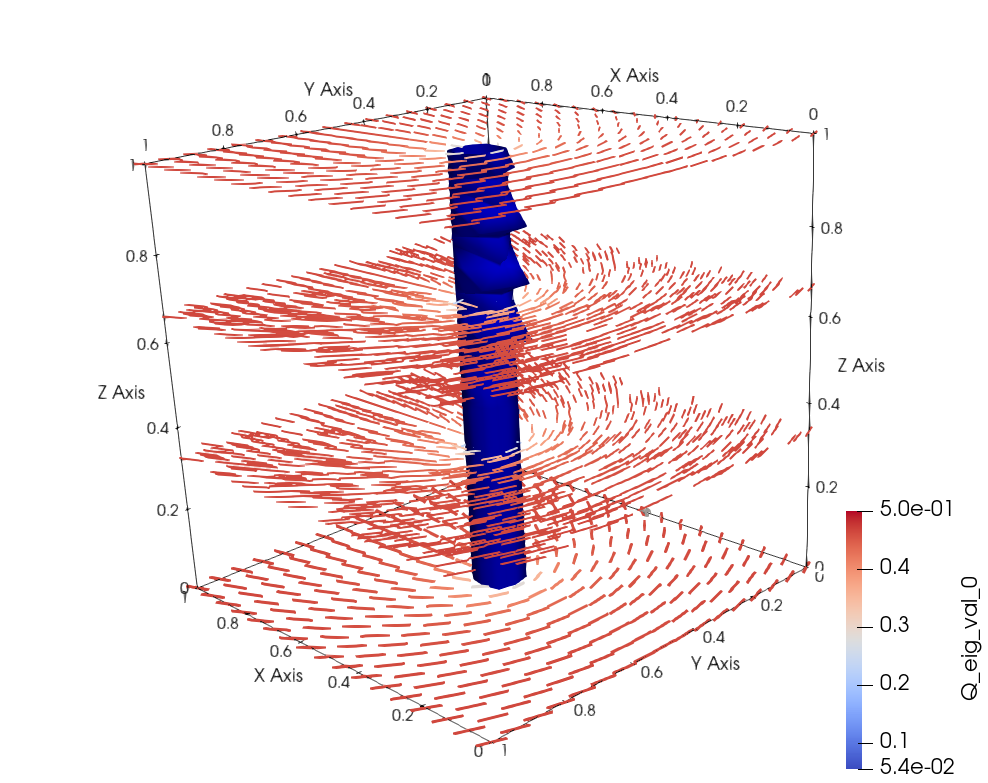}
\end{subfigure}
\begin{subfigure}{.32\textwidth}
\centering
\includegraphics[width=.99\linewidth]{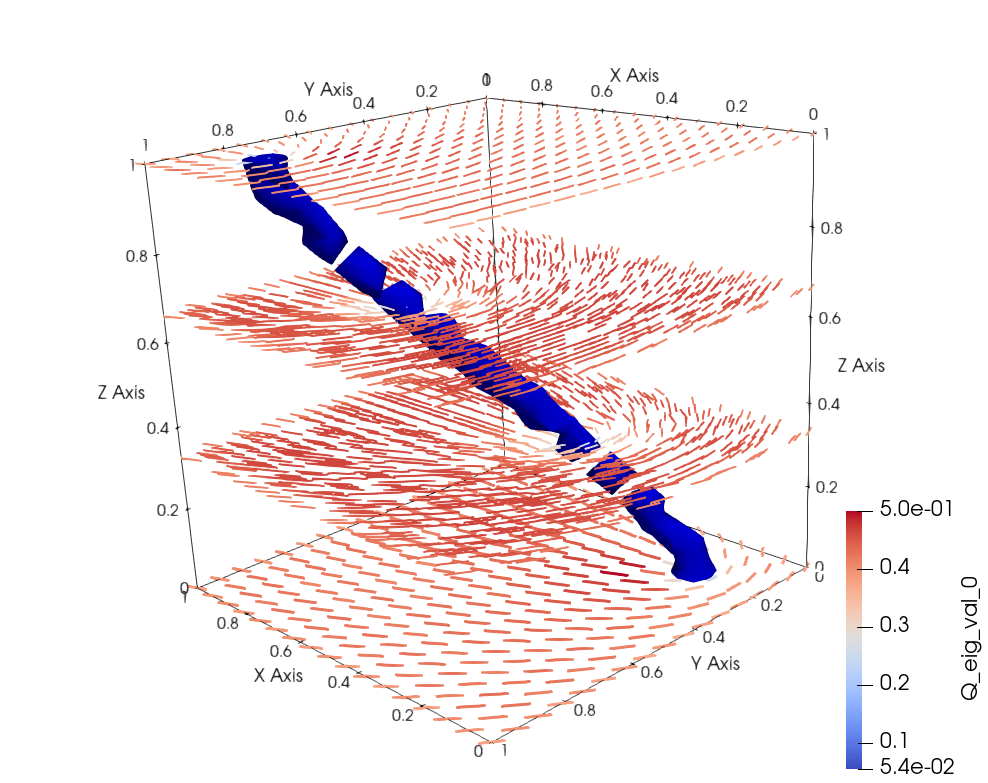}
\end{subfigure}
\begin{subfigure}{.32\textwidth}
\centering
\includegraphics[width=.98\linewidth]{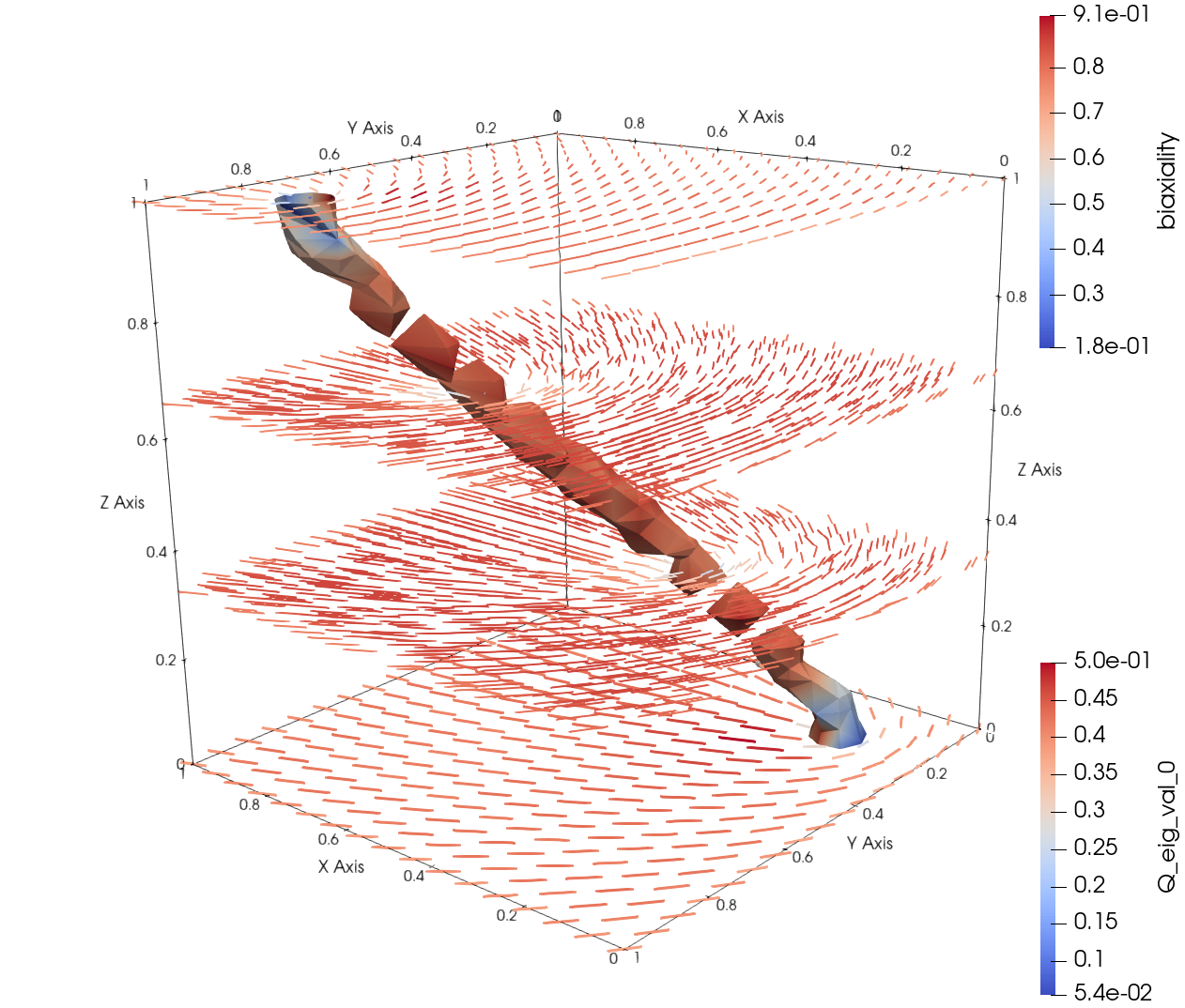}
\end{subfigure}

\caption{Initial state for $\tQ(t=0)$ (left) and final state for $\tQ(t=\tfinal)$ (middle) (\cref{sec:ctrl_+1/2_curve_defect_3D}). 
Line segments correspond to the eigenvector of $\tQ$ with maximum eigenvalue; the color scale is based on the maximum eigenvalue.  
The position of the line defect (at the final time) is very close to that in the target.  {Right plot: the color scale corresponds to the biaxiality measure (see \cref{rem:biaxiality}).  Away from the defect (not shown), $\tQ$ is essentially uniaxial with $\beta_{\mathrm{bi}}(\tQ) \approx 0$.}
}
\label{fig:ctrl_+1/2_curve_defect_3D_Q_0_tf}
\vspace{-0.3cm}
\end{figure}

\begin{remark}\label{rem:biaxiality}
{In dimension $d=2$, all $\tQ$-tensors have a uniaxial form.  For $d=3$, $\tQ$ is uniaxial if and only if $\tQ$ has two repeated eigenvalues \cite{Virga_book1994}.  Moreover, even if the initial condition $\tQ_{0}$ is uniaxial, the solution $\tQ(t,\cdot)$ of \cref{eq:forward_problem} will not be uniaxial in general, i.e. it will become biaxial with three distinct eigenvalues.  Typically, the solution is only biaxial near a defect; away from a defect, it is (essentially) uniaxial because of the global minimum properties of the bulk potential in \cref{eqn:Landau-deGennes_bulk_potential} (see \cite{Mottram_arXiv2014}).}

{The right plot of \cref{fig:ctrl_+1/2_curve_defect_3D_Q_0_tf} shows the normalized biaxiality measure \cite{Majumdar_ARMA2010}, $\beta_{\mathrm{bi}}(\tQ)$, of the solution $\tQ$ on the iso-surface surrounding the line defect.  Note that $0 \leq \beta_{\mathrm{bi}}(\tQ) \leq 1$, with $\beta_{\mathrm{bi}}(\tQ) = 0$ corresponding to a uniaxial state and $\beta_{\mathrm{bi}}(\tQ) = 1$ representing ``maximal'' biaxiality.}
\end{remark}

%---------------------------------------------------------------------------
\section{Conclusions}\label{sec:conclusion}
%---------------------------------------------------------------------------

The main contribution of this paper is to show that optimal control of LC devices, in the framework of the LdG model, is possible.  Indeed, our numerical study demonstrates this effectively by directly controlling the placement of defects, which is of considerable interest in the LC scientific community.  We only consider boundary controls in our numerical study since that is most relevant in applications.  
Further extensions of our framework, as related to actual LC systems, would involve controls that are either finite-dimensional (in space) or with a special restriction on the admissible controls, e.g. homeotropic versus planar anchoring for boundary controls.

From an analytical standpoint, by restricting our study to the one-parameter model (i.e. the only non-zero elastic constant is $\Li_{1} = 1$), we were able to exploit a large number of derivation techniques for the optimal control of scalar Allen-Cahn equations. The rigorous proofs for the bounds and energy estimates in the tensor-valued setting have therefore been relegated to appendices. Nevertheless, there remain a number of analytical challenges if we wish to go beyond the one-parameter model.  
For example, our current proof of continuity in space-time may
only work in the current setting and new techniques or regularity results also appear necessary.  This is because for more general elastic constants, the Laplacian in \cref{eq:forward_problem} is replaced by a more general elliptic operator that fully couples all components of the $\tQ$-tensor.

Finally, our numerical study made use of a  basic optimization algorithm.  A more advanced scheme, e.g., one based on second-order information would require an additional sensitivity result to derive an analytical formula for second-order directional derivatives (Hessian-vector products) for use in Newton-type methods.  At least for the bulk energy term considered here, such a result should be obtainable by modifying the proof of differentiability in \cref{sec:first_order}.

\appendix

\vspace{-0.1cm}

\section[Proof of Thm 4]{Proof of \cref{prop:disc-w-estm}}\label{app:prop.3.7}
We start by using the test function $\tP = \tQ^{n}_{t}(s)$ for all $s \in [0,\tfinal]$ in \eqref{eq:weak_discrete}. This leads to
\vspace{-0.2cm}
\begin{multline}\label{eq:sec_apriori_est_1}
\|\tQ^{n}_{t}(s)\|^2_{\Hs}
+ 
\frac{d}{dt} \| \nabla \tQ^{n}(s)\|^2_{\Hs}
+
\frac{\anchcoef}{2} \frac{d}{dt} \|\tQ^{n}(s)\|^2_{\HGs} + \frac{1}{\bulketa^2} \frac{d}{dt}  \int_{\Omega} \bulkfunc(\tQ^{n}(s)) \\
= \domcoef \inner{\domcon(s)}{\tQ^{n}_{t}(s)}{\Hs} + \anchcoef \inner{\bdycon(s)}{\tQ^{n}_{t}(s)}{\HGs}.
\end{multline}
{\color{black}
%We simplicity define $M_2 : \R^4 \to \R$ for arbitrary  $\delta_1,\dots,\delta_4 > 0$ by 
%\vspace{-0.2cm} 
%\[
%M_2(x_1,x_2,x_3,x_4) = 
%\delta_1x_1
%+
%\delta_2 x_2
%+
%\delta_3 x_3
%+
%\delta_4 x_4.
%\]
%for arbitrary positive constants $\delta_1,\dots,\delta_4$.
%We will use 
%$
%M_2 = M_2(
%\| \bdycon(t) \|^2_{\HGs},
%\| \bdycon(0)\|^2_{\HGs},
%\|(\bdycon)_t \|^2_{L^2(\bdycyl_t)}, 
%\| \domcon \|^2_{L^2(\domcyl_t)}
%)
%$
%below.
%\[
%\frac{\delta_1}{2} \| \bdycon(t) \|^2_{\HGs}  
%+
% \frac{\delta_2}{2} \| \bdycon(0) \|^2_{\HGs} 
%+
%\frac{\delta_3}{2} \| (\bdycon)_t \|^2_{L^2(\bdycyl_t)}
%\le
%c_{1}
%\]
%there exists $c_{1} > 0 $ (independent of $n$) such that 
%\[
%\frac{\delta_1}{2} \| \bdycon(t) \|^2_{\HGs}  
%+
% \frac{\delta_2}{2} \| \bdycon(0) \|^2_{\HGs} 
%+
%\frac{\delta_3}{2} \| (\bdycon)_t \|^2_{L^2(\bdycyl_t)}
%\le
%c_{1}
%\]
We continue \eqref{eq:sec_apriori_est_1} by using
$\bdycon \in H^1(0,\tfinal;\HGs)$,
integrating from  $0$ to $t \in (0,\tfinal]$, and rearranging terms
to obtain new constants $\delta'_1,\dots,\delta'_4 > 0$:
}
%Returning to \eqref{eq:sec_apriori_est_1}, 
%\vspace{-0.2cm}
%\begin{multline}\label{eq:sec_apriori_est_2}
%\|\tQ^{n}_{t}\|^2_{L^2(\domcyl_{t})}
%+ 
% \| \nabla \tQ^{n}(t)\|^2_{\Hs}
% -
%  \| \nabla \tQ^{n}_{0}\|^2_{\Hs}
%+
%\frac{\anchcoef}{2} \|\tQ^{n}(t)\|^2_{\HGs}
%-
%\frac{\anchcoef}{2} \|\tQ^{n}_{0}\|^2_{\HGs} \\ 
%+ \frac{1}{\bulketa^2}\int_{\Omega} \bulkfunc(\tQ^{n}(t))
%-
%\frac{1}{\bulketa^2}\int_{\Omega} \bulkfunc(\tQ^{n}_{0})\\
%%\le \domcoef \int_{0}^t \inner{\domcon}{\tQ^{n}_{t}}{\Hs} + 
%%\anchcoef \int_{0}^t \inner{\bdycon}{\tQ^{n}_{t}}{\HGs} \\
%%\frac{\delta_4\domcoef}{2} \| \domcon \|^2_{L^2(\domcyl_{t})} 
%%+ 
%%c_1
%\le M_2
%%(
%%\| \bdycon(t) \|^2_{\HGs},
%%\| \bdycon(0)\|^2_{\HGs},
%%\|(\bdycon)_t \|^2_{L^2(\bdycyl_t)}, 
%%\| \domcon \|^2_{L^2(\domcyl_t)}
%%)
%+\frac{\domcoef}{2 \delta_4} \| \tQ^{n}_{t} \|^2_{L^2(\domcyl_{t})}
%+
%\frac{1}{2\delta_1} \|\tQ^{n}(t) \|^2_{\HGs} +
%\frac{1}{2\delta_2} \|\tQ^{n}_{0} \|^2_{\HGs} + 
%\frac{1}{2\delta_3} \|\tQ^{n}\|^2_{L^2(\domcyl_{t})}.
%%\frac{\delta_1 \domcoef}{2} \|\tQ^{n}_{t}\|^2_{L^2(\domcyl_{t})} 
%%+ 
%%\frac{\delta_2 \anchcoef}{2} \|\tQ^{n}_{t}\|^2_{L^2(\bdycyl_t)} +  c_u,
%\end{multline}
%where $\delta_4 > 0$ is arbitrary and independent of $n$.
%As with $M_1$ above, we will leave off the arguments of $M_2$ when it is clear in context.
%Next, by appropriately choosing each $\delta_i > 0$, we obtain new constants $\delta'_i > 0$, $i =1,\dots,4$ such that
{\color{black} 
%We then rearrange terms and obtain new constraints $\delta'_1,\dots,\delta'_4$:
}
\vspace{-0.4cm}
\begin{multline}\label{eq:sec_apriori_est_3}
\delta'_{4} \|\tQ^{n}_{t}\|^2_{L^2(\domcyl_{t})}
+ 
 \| \nabla \tQ^{n}(t)\|^2_{\Hs}
+
\frac{1}{\bulketa^2}\int_{\Omega} \bulkfunc(\tQ^{n}(t))
+
\delta'_1\|\tQ^{n}(t)\|^2_{\HGs}
\le\\
  \| \nabla \tQ^{n}_{0}\|^2_{\Hs} 
+
\frac{1}{\bulketa^2}\int_{\Omega} \bulkfunc(\tQ^{n}_{0}) 
+
\delta'_2 \|\tQ^{n}_{0}\|^2_{\HGs}  
+
\delta'_3 \| \tQ^{n}\|^2_{L^2(\domcyl_{t})}
+  
M_2
%(
%\| \bdycon(t) \|^2_{\HGs},
%\| \bdycon(0)\|^2_{\HGs},
%\|(\bdycon)_t \|^2_{L^2(\bdycyl_t)}, 
%\|u\|^2_{L^2(\domcyl_t)}
%)
.
\end{multline}
{\color{black}
We can bound the penultimate term in \eqref{eq:sec_apriori_est_3} by  
%This bound can be extended by 
applying \eqref{eq:first_apriori_5} and
\vspace{-0.2cm}
\[
\|\tQ^{n}\|_{ L^2(\domcyl_{t})}^2
\le 
\| \tQ^{n}\|_{ L^2(\domcyl) }^2
\le 
c_{\rm emb} 
\| \tQ^{n}\|_{L^2(0,\tfinal;\Vs)}^2
\le
c_{\rm emb} M_1,
\]
where $c_{\rm emb}$ is an embedding constant. 
%This gives us a new constant $c'_u > 0$ such that 
%\[
%\delta'_3 \| \tQ^{n}\|^2_{L^2(\domcyl_{t})}
%+  c_u
%\le 
%c'_u,
%\]
%independently of $n$.
%This gives us the bound 
%%\vspace{-0.2cm}
%$
%\delta'_3 \| \tQ^{n}\|^2_{L^2(\domcyl_{t})}
%\le
%M_1,
%$
%where 
$M_1$  absorbs $\delta'_3$ and $c_{\rm emb}$ below. }

Based on the order of the nonlinearity $\bulkfunc$, the continuity of the trace operator, and the convex splitting $\bulkfunc = \bulkimp - \bulkexp$, there is a constant $M_0 \ge 0$ such that
\vspace{-0.2cm}
\[
  \| \nabla \tQ^{n}_{0}\|^2_{\Hs} 
+
\frac{1}{\bulketa^2}\int_{\Omega} \bulkimp(\tQ^{n}_{0}) 
- \frac{1}{\bulketa^2}\int_{\Omega} \bulkexp(\tQ^{n}_{0}) 
+
\delta'_{2}\|\tQ^{n}_{0}\|^2_{\HGs}  
\le 
M_0.
\]
%for all $n \in N$ sufficiently large.  
%then using that $\bulkimp$ is convex (so is weak l.s.c.) and that $\bulkexp(y)$ is quadratic in $y$, we can take the limit.
{\color{black}
Since $\bulkimp$ is continuous on $\Vs$ due to the Sobolev embedding theorem and $\bulkexp(\tQ)$ is quadratic in $\tQ$, we can pass to the limit in $n$ and thus obtain \eqref{eq:M0-bound}.}

Next, since $\bulkfunc$ is bounded from below, we can adjust all the constants and coefficients if necessary  %$c_u$ or $c_y$ if necessary 
to obtain the bound
\vspace{-0.2cm}
\begin{equation}\label{eq:sec_apriori_est_4}
\delta'_4 \|\tQ^{n}_{t}\|^2_{L^2(\domcyl_{t})}
+ 
 \| \nabla \tQ^{n}(t)\|^2_{\Hs}
+
\delta'_{1} \|\tQ^{n}(t)\|^2_{\HGs}
\le
M_0 + M_1 + M_2.
%c'_y
%+  c'_u.
\end{equation}
This yields \eqref{eq:sec_apriori_est_4-supp}.
Now, by letting $\varepsilon > 0$ be a small positive constant, we can bound \eqref{eq:sec_apriori_est_4} from below, which yields
\vspace{-0.3cm}
\begin{equation}\label{eq:sec_apriori_est_5}
\delta'_4 \|\tQ^{n}_{t}\|^2_{L^2(\domcyl_{t})}
+ 
 \min\{1 - \varepsilon \delta'_1,\varepsilon \frac{\delta'_1}{\kappa_0}\}
 \left(
 \| \nabla \tQ^{n}(t)\|^2_{\Hs}
+
\|\tQ^{n}(t)\|^2_{\Hs}
\right)
\le
M_0 + M_1 + M_2.
%c'_y
%+  c'_u.
\end{equation}
Here, $\kappa_0$ comes from using a Poincar\'{e} type inequality.  
%is the constant from \eqref{eq:bdy_blk}.  
We can now adjust the coefficients and constants to deduce the bound:
\vspace{-0.2cm}
\begin{equation}\label{eq:sec_apriori_est_6}
\|\tQ^{n}_{t}\|^2_{L^2(\domcyl_{t})}
+ 
 \| \nabla \tQ^{n}(t)\|^2_{\Hs}
+
\|\tQ^{n}(t)\|^2_{\Hs}
\le c \left( M_0 + M_1 + M_2 \right).
\end{equation}
This yields \eqref{eq:sec_apriori_est_6-supp}.
It follows from \eqref{eq:sec_apriori_est_4}, \eqref{eq:sec_apriori_est_6}, and \eqref{eq:first_apriori_5} that $\left\{ \tQ^{n} \right\}$ is uniformly bounded in $\bigW$ \eqref{eq:space_of_weak_soln}.

\section[Proof of Thm 6]{Proof of \cref{thm:exist_state}}\label{app:thm.3.9}
The Aubin-Lions-Simon Lemma, see e.g., Theorem II.5.16, pages 102-103 in \cite{FBoyer_PFabrie_2013} provides several helpful statements. 
We provide brief justifications afterwards, as these are well-known embeddings.
\begin{enumerate}
\item There exists a subsequence $\{ \tQ^{k} \}$ with $\tQ^{k} := \tQ^{n_k}$ that converges strongly to the function $\bar{\tQ}$ in $C([0,\tfinal]; L^{6 - \varepsilon}(\Omega,\symmtraceless))$ for $\varepsilon \in (0,5]$.
\item There exists a subsequence $\{ \tQ^{l} \}$ with $\tQ^{l} := \tQ^{k_l}$ that converges weakly in
\vspace{-0.2cm}
\[
\bigW_1 := \left\{ \tQ \in L^2(0,\tfinal;\Vs) \left|\; \tQ_t \in L^2(0,\tfinal;\Hs) \right.\right\} ~ \text{ to } ~ \bar{\tQ}.
\]
\item There exists a subsequence $\{\tQ^{m}\}$ with $\tQ^{m} := \tQ^{l_m}$ that converges weakly to $\bar{\tQ}$ in $L^2(0,\tfinal;\HGs)$.
\end{enumerate}
The first subsequence exists by the Aubin-Lions Lemma, which implies that
\vspace{-0.2cm}
\[
\bigW_2 := \left\{ \tQ \in L^{\infty}(0,\tfinal;\Vs) \left|\; \tQ_t \in L^2(0,\tfinal;\Hs) \right.\right\},
\]
is compactly embedded into the space $C([0,\tfinal]; L^{6 - \varepsilon}(\Omega,\symmtraceless))$.  Here, we make use of the Sobolev embedding theorem to embed $\Vs$ into $L^{6 - \varepsilon}(\Omega,\symmtraceless)$. The second subsequence exists due to the reflexivity of $\bigW_1$; likewise for the final subsequence. 
%In addition, we can also argue that $\bar{\tQ} \in \bigW$, which we will discuss next. 
{\color{black}
Finally, we can also argue that $\bar{\tQ} \in \bigW$, by appealing to the bounds in \cref{app:prop.3.7}, which are stable under passage to the limit in $n$.  To be more specific, we can find an independent constant $\rho > 0$ such that
\vspace{-0.3cm}
\begin{equation}\label{eq:asym_1}
\aligned
\mathrm{ess\, sup}_{t \in [0,\tfinal]} (\| \nabla \bar{\tQ}(t)\|^2_{\Hs}
+
\|\bar{\tQ}(t)\|^2_{\Hs})^{1/2} &\le \rho,\\
\mathrm{ess\, sup}_{t \in [0,\tfinal]}\|\bar{\tQ}(t)\|_{\HGs} &\le \rho,\\
(\| \bar{\tQ} \|^2_{L^2(\domcyl)} + \| \bar{\tQ}_t \|^2_{L^2(\domcyl)})^{1/2} &\le \rho\\
\| \bar{\tQ} \|_{L^2(0,\tfinal;\Vs)} &\le \rho.
\endaligned
\end{equation}
We arrive at the bound in the space $\bigW$, i.e. \eqref{eq:asym_5}.
}

It remains to show that $\bar{\tQ}$ is a weak solution of \eqref{eq:forward_problem}. Uniqueness is a consequence of \Cref{thm:soln_cont}.
For arbitrarily fixed data $(\domcon,\bdycon,\tQ_{0})$ that satisfies \eqref{eq:minimal_input_reg} and a test function $\tP \in \bigW_1$ such that $\tP(0) = 0$ a.e., 
%\blue{SWW: you mean $\tP(0) = 0$??} 
we recall \eqref{eq:weak_discrete} and integrate in $t$:
\vspace{-0.2cm}
\begin{equation}\label{eq:weak_discrete_a}
\begin{split}
\int_{0}^{\tfinal} \inner{\tQ^{n}_{t}}{\tP}{\Hs}
+ &
\int_{0}^{\tfinal} \inner{\nabla \tQ^{n}}{\nabla \tP}{\Hs}
+ 
\frac{1}{\bulketa^2} \int_{0}^{\tfinal} \int_{\Omega} \bulkfunc'(\tQ^{n}) \dd \tP
\\
+ \anchcoef & \int_{0}^{\tfinal} \inner{\tQ^{n}}{\tP}{\HGs} 
= \domcoef \int_{0}^{\tfinal} \inner{\domcon}{\tP}{\Hs} + \anchcoef \int_{0}^{\tfinal} \inner{\bdycon}{\tP}{\HGs}.
\end{split}
\end{equation}
{\color{black}
The convergence of the linear terms in \cref{eq:weak_discrete_a} follows by straightforward arguments, e.g., weak convergence and use of compact embeddings. For the nonlinear term, it suffices to note that 
$\psi'$ is globally Lipschitz, which provides, e.g., strong convergence in $C([0,t_f];L^{7/4}(\Omega))$ of $
\bulkfunc'(\tQ^{k})$ to $\bulkfunc'(\bar{\tQ})$. Given $\tP \in L^2([0,\tfinal];\Vs) \subset L^2([0,\tfinal]; L^{7/3}(\Omega))$, we can then pass to the limit.
}

\section[Proof of Thm 7]{Proof of \cref{thm:higher_reg_thm}}\label{app:thm.3.11}
After using the bootstrapping and decomposition technique, we derive from \cite[Thm. 5.5]{Troltzsch_book2010} the existence of some constant $c > 0$, independent of $(\domcon,\bdycon)$, for which we have
\vspace{-0.2cm}
\begin{equation}\label{eq:higher_reg_1}
\| \bar{\tQ} \|_{C(\overline{\domcyl})} \le c\left(
\| \domcoef \domcon- \frac{1}{\bulketa^2} \psi(\bar{\tQ}) \|_{L^r(\domcyl)} + 
\| \anchcoef \bdycon \|_{L^s(\bdycyl)} + 
\| \tQ_{0} \|_{C(\overline{\Omega})}
\right).
\end{equation}
We remove the dependence on $\bar{\tQ}$ from the right-hand side, by noting that \eqref{eqn:mod_LdG_bulk_pot_bnds} implies
\vspace{-0.3cm}
\begin{equation}\label{eq:higher_reg_2}
\|\psi'(\bar{\tQ}) \|_{L^r(\domcyl)} \le 
c_1 \| \bar{\tQ} \|_{L^r(\domcyl)} \le
c_1 c_{\rm emb}\| \bar{\tQ} \|_{L^{\infty}(0,\tfinal;\Vs)}
\le 
\rho,
\end{equation}
where the second inequality follows from the continuous embedding of $L^{\infty}(0,\tfinal;\Vs)$ into $L^r(\domcyl)$ (provided $r \in (5/2,6]$), $c_{\rm emb}$ is the associated embedding constant, and $\rho$ is from \eqref{eq:asym_1}. Next, we derive an explicit bound for $\rho$. Starting from \eqref{eq:sec_apriori_est_6}, we note that before passing to the limit in $n$ we have
\vspace{-0.2cm}
\begin{multline*}
\sqrt{
 \| \nabla \tQ^{n}(t)\|^2_{\Hs}
+
\|\tQ^{n}(t)\|^2_{\Hs}}
\le
%M'.
\sqrt{M_0 + M_1 + M_2}
\le
\sqrt{M_0} + \\
M_1(\| \tQ^{n}_{0} \|_{\Hs}, \|\domcon\|_{L^2(\domcyl)}, \|\bdycon \|_{L^2(\bdycyl)}) + 
M_2(\|\bdycon(t) \|_{\HGs},\|\bdycon(0)\|_{\HGs}, \|(\bdycon)_t\|_{L^2(\bdycyl_t)},\|u\|_{L^2(\domcyl_t)}).
\end{multline*}
Here, we use the subadditivity of $\sqrt{\cdot}$ along with the fact that $M_1$ and $M_2$ are simple multilinear maps of their arguments with positive coefficients. We may then pass to the limit in $n$ along an appropriate subsequence and obtain the same inequality independent of $n$. The $M_0$-term is independent of $(\domcon,\bdycon)$.  The $M_1$-term can be bounded in the first argument by $\| \tQ_{0} \|_{C(\overline{\Omega})}$ and the $(\domcon,\bdycon)$-terms by stronger norms. For the $M_2$-term we have several possibilities. Since $(\domcon,\bdycon) \in \conreg$, with norm given by \cref{eq:def_norm_UUg},  
%\[
%\| (\domcon,\bdycon) \|_{\conreg} = 
%\max\{ 
%\| (\domcon,\bdycon) \|_{L^2(\domcyl) \times H^1(0,\tfinal;\HGs)},
%\| (\domcon,\bdycon) \|_{L^{r}(\domcyl) \times L^{s}(\bdycyl)}
%\},
%\]
and $H^1(0,\tfinal;\HGs)$ is continuously embedded into $C([0,\tfinal];H_{\Gm})$, we can bound the first two arguments in $M_2$ first by $\| \bdycon \|_{C([0,\tfinal];H_{\Gm})}$ and then further from above by $\| \bdycon \|_{H^1(0,\tfinal;\HGs)}$. The latter two arguments can be bounded from above by the norms
$\| (\bdycon)_t \|_{L^2(\bdycyl)}$ and $\|\domcon\|_{L^2(\domcyl)}$, respectively.  
Clearly, the third argument can be bounded from above by $\| \domcon \|_{H^1(0,\tfinal;\HGs)}$. Since $r > 2$, the fourth argument  can be bounded from above by $\| \domcon \|_{L^r(\domcyl)}$. By combining all of these observations, we deduce the existence of a constant $c > 0$, independent of $\bar{\tQ}, (\domcon,\bdycon)$ and $\tQ_{0}$ such that for all $t \in [0,\tfinal]$ we have
\vspace{-0.2cm}
\[
\sqrt{
 \| \nabla \bar{\tQ}(t)\|^2_{\Hs}
+
\|\bar{\tQ}(t)\|^2_{\Hs}}
\le
\sqrt{M_0} + c (\| (\domcon,\bdycon) \|_{\conreg} + \| \tQ_{0} \|_{C(\overline{\Omega})}),
\]
which implies that 
$
\| \bar{\tQ} \|_{L^{\infty}(0,\tfinal;\Vs)} \le \sqrt{M_0} + c (\| (\domcon,\bdycon) \|_{\conreg} + \| \tQ_{0} \|_{C(\overline{\Omega})}).
$
Combining this bound with \eqref{eq:higher_reg_1} and \eqref{eq:higher_reg_2}, there exists a constant $c > 0$, independent of $\bar{\tQ}, (\domcon,\bdycon)$ and $\tQ_{0}$, such that 
$\| \bar{\tQ} \|_{C(\overline{\domcyl})} \le c\left(\sqrt{M_0} 
 + 
 \| (\domcon,\bdycon) \|_{\conreg}
+
\| \tQ_{0} \|_{C(\overline{\Omega})}
\right)$.  The assertion then follows.

%\begin{lemma}
%Test Lemma.
%\end{lemma}

%
%\section*{Acknowledgments}
%We would like to acknowledge the assistance of volunteers in putting
%together this example manuscript and supplement.

\vspace{-0.2cm}

\bibliographystyle{siamplain}
\bibliography{MasterBibTeX}
%\bibliography{C:/FILES/LaTex/Master_BIBTEX/MasterBibTeX}

\end{document}